\documentclass[reqno,11pt]{amsart}

\usepackage{xcolor}
\usepackage{graphicx}
\usepackage[hidelinks]{hyperref}
\usepackage{amscd}

\usepackage{amsmath}
\usepackage{amsthm}
\usepackage{amssymb}
\usepackage{latexsym,array}
\usepackage{amsfonts}
\usepackage{shadow}
\usepackage{amsbsy}
\usepackage{amssymb}
\usepackage{dsfont}
\usepackage{mathtools}
\usepackage{enumitem}
\usepackage{doi}

\usepackage[notref, notcite, final]{showkeys}
\usepackage{color}
\usepackage{graphicx}
\usepackage[hidelinks]{hyperref}
\usepackage{cleveref}
\usepackage{fullpage}
\usepackage{comment}
\usepackage{tikz-cd}
\usepackage{tikz}
\usepackage{float}

\usepackage{dsfont}

\usepackage{cleveref}

\numberwithin{equation}{section}

\newcommand{\id}{\mathcal{I}}

\newtheorem{theorem}{Theorem}[section]
\newtheorem{lemma}[theorem]{Lemma}

\newtheorem{proposition}[theorem]{Proposition}

\theoremstyle{definition}
\newtheorem{remark}[theorem]{{\bf Remark}}
\newtheorem{definition}[theorem]{Definition}

\newcommand{\cc}{\mathbb{C}}

\newcommand{\hh}{\mathbb{H}}

\newcommand{\Q}{\mathcal{Q}}

\crefname{enumi}{}{}
\crefname{enumii}{}{}

\title[]{
Harmonic and polyanalytic functional calculi on the $S$-spectrum for unbounded operators}

\date{}

\author[Fabrizio Colombo]{Fabrizio Colombo}
\address{(FC) Politecnico di Milano, Dipartimento di Matematica, Via E. Bonardi, 9, 20133 Milano, Italy}
\email{fabrizio.colombo@polimi.it}

\author[A. De Martino]{Antonino De Martino}
\address{(ADM)
 Schmid College of Science and Technology,  Chapman University One University Drive Orange\\
California 92866,
USA
}\email{ademartino@chapman.edu}

\author[S. Pinton]{Stefano Pinton}
\address{(SP)
Politecnico di Milano\\Dipartimento di Matematica\\Via E. Bonardi, 9\\20133
Milano, Italy
} \email{stefano.pinton@polimi.it}

\begin{document}
	
	\maketitle
	
\begin{abstract}
Harmonic and polyanalytic functional calculi have been recently
defined for bounded commuting operators.
Their definitions are based on the Cauchy formula of slice hyperholomorphic functions and on
the factorization of the Laplace operator
in terms of the Cauchy-Fueter operator
 $\mathcal{D}$ and of its conjugate $\overline{\mathcal{D}}$.
Thanks to the Fueter extension theorem when we apply the operator $\mathcal{D}$ to
slice hyperholomorphic functions we obtain harmonic functions
and via the Cauchy formula of slice hyperholomorphic functions we establish
an integral representation for harmonic functions.
This integral formula is used to define the harmonic functional calculus on the $S$-spectrum.
Another possibility is to apply the conjugate of the Cauchy-Fueter operator to slice hyperholomorphic functions.
In this case, with a similar procedure we obtain the class of polyanalytic functions, their integral representation and the associated polyanalytic functional calculus.
The aim of this paper is to extend the harmonic and the polyanalytic functional calculi
to the case of unbounded operators and to prove some of the most important properties.
These two functional calculi belong to so called fine structures on the $S$-spectrum in the quaternionic setting.
Fine structures on the $S$-spectrum associated with Clifford algebras
constitute a new research area that deeply connects different research fields such as operator theory, harmonic analysis and hypercomplex analysis.
\end{abstract}



\tableofcontents

\section{Introduction}
One of the main motivations to investigate quaternionic spectral theory
 can be found in the paper of G. Birkhoff and J. Von Neumann, see \cite{BF},
 where the authors show that quantum mechanics can also be formulated using quaternions,
  but they do not specify the definition of spectrum for quaternionic linear operator.
The spectral theory on the $S$-spectrum for quaternionic linear operators began
 in 2006  with the discovery of the $S$-spectrum.
 This notion of spectrum  was identified
 using only methods in hypercomplex analysis even though its existence was suggested
 by quaternionic quantum mechanics, for more details see the introduction of the book \cite{CGK}.
  The notion of $S$-spectrum extends also to operators in the
 Clifford algebra setting, see \cite{ColomboSabadiniStruppa2011,ColSab2006},
 but recently it has been shown that the quaternionic and the Clifford settings
 are just particular cases of a more general framework in which the spectral theory
 on the $S$-spectrum can be developed, see \cite{ADVCGKS,PAMSCKPS}.
  Using the notion of $S$-spectrum
   it was possible to prove the spectral theorem for quaternionic linear operators, see \cite{ACK},
   that is a central theorem for the formulation of quantum mechanics,
    and more recently the spectral theorem was extended to Clifford operators,
   see \cite{ColKim}.

The spectral theory based on the $S$-spectrum is systematically organised in the books
\cite{FJBOOK, CGK, ColomboSabadiniStruppa2011}. Nowadays, this theory has several research directions, without clamming completeness we mention: the slice hyperholomorphic Schur analysis, see \cite{ACS2016}, new classes of fractional diffusion problems based on fractional powers of quaternionic linear operators, see \cite{CGK, ColomboDenizPinton2020, ColomboDenizPinton2021, ColomboPelosoPinton2019}. The results on the fractional powers are based on the $H^{\infty}$-functional calculus, see \cite{ACQS2016, CGdiffusion2018}. Other interesting research directions are the study of the characteristic operator function, see \cite{AlpayColSab2020}, and the quaternionic perturbation theory and invariant subspaces, see \cite{CereColKaSab}.

\medskip
In recent times it has been developed a new branch of the spectral theory on the $S$-spectrum that is called fine structures on the $S$-spectrum.
It consists of the function spaces arising form the Fueter-Sce theorem with theirs integral representations that are used to define various functional calculi, see \cite{CDPS1,Fivedim,Polyf1,Polyf2}.

The function spaces of a given fine structure
are determined by the factorizations of the second operators $T_{FS2}$ in the Fueter-Sce extension theorem and on the Cauchy formula of slice hyperholomorphic functions.
The integral representations of the spaces of the fine structures
are crucial to define the associated functional calculi
 for bounded operators on the $S$-spectrum. Among these functional calculi we have the harmonic, the poly-harmonic, the poly-analytic ones and several others. In this paper we concentrate on the quaternionic setting
that contains two important fine structures.
In fact, the second mapping  $T_{FS2}$ in the Fueter extension theorem  is equal to the   Laplace operator $\Delta$ in dimension $4$ and the most important factorizations of $\Delta$ lead to the
  harmonic fine structure and to the polyanalytic one.

\medskip
{\em The goal of this paper is to
 extend the harmonic and the polyanalytic functional calculi to unbounded operators and to prove some of their properties.}

\medskip
In order to present our results we need to explain more about the setting in which we work,
so we will quote some results that we will need in the sequel and
whose proofs can be found in \cite{CGK}. We start by fixing the notation of the quaternions, that are defined as
$$ \mathbb{H}:= \{q=q_0+q_1e_1+q_2e_2+q_3e_3 \,| \, q_0, q_1, q_2, q_3 \in \mathbb{R}\},$$
where the imaginary units satisfy the relations
$$ e_1^2=e_2^2=e_3^2=-1, \quad \hbox{and} \quad e_1e_2=-e_2e_1=e_3, \, e_2e_3=-e_3e_2=e_1, \, e_3e_1=-e_1e_3=e_2.$$
We denote by $ \hbox{Re}(q)=q_0$ the real part of a quaternion and by $ \underline{q}:= q_1e_1+q_2e_2+q_3e_3$ its imaginary part. The conjugate of a quaternion $q \in \mathbb{H}$ is defined as $ \bar{q}=q_0- \underline{q}$ and the modulus of $q  \in \mathbb{H}$ is given by $|q|= \sqrt{q \bar{q}}= \sqrt{q_0^2+q_1^2+q_2^2+q_3^2}.$ The unit sphere of purely imaginary quaternions is defined as
$$ \mathbb{S}:= \{\underline{q}=q_1e_1+q_2e_2+q_3e_3\, | \, q_1^2+q_2^2+q_3^2=1\}.$$
We observe that if $ J \in \mathbb{S}$ then $J^2=-1$. This means that $J$ behaves like an imaginary unit, and we denote by
$$ \mathbb{C}_J:= \{u+Jv \, | \, u,v \in \mathbb{R}\}$$
an isomorphic copy of the complex numbers. If we consider a non real quaternion $q=q_0+ \underline{q}=q_0+ J_q |\underline{q}|$, where $J_q:= \underline{q}/ |\underline{q}|$. We can associate to $q$ the 2-sphere defined by
$$ [q]:= \{q_0+J |\underline{q}| \, | \, J \in \mathbb{S}\}.$$
We say that $U \subset \mathbb{H}$ is axially symmetric if, for every $u+Iv \in U$, all the elements $u+Jv$ for $J \in \mathbb{S}$ are contained in $U$. Furthermore, the set $U$ is called slice domain if $U \cap \mathbb{R} \neq \emptyset$ and if $U \cap \mathbb{C}_J$ is a domain in $ \mathbb{C}_J$ for every $J \in \mathbb{S}$.
\newline
\newline
Axially symmetric sets are suitable domains of the following class of functions.
\begin{definition}[Slice hyperholomorphic functions]
Let $U \subset \mathbb{H}$ be an axially symmetric open set and let
$$ \mathcal{U}= \{(u,v) \in \mathbb{R}^2 \, | \, u+ \mathbb{S}v \in U\}.$$
We say that a function $f:U \to \mathbb{H}$ of the form
$$ f(q)=f(u+Jv)= \alpha(u,v)+J \beta(u,v) \qquad \left({\rm resp.} \ f(q)=f(u+Jv)=\alpha(u,v)+ \beta(u,v)J\right)$$
is left (resp. right) slice hyperholomorphic if $ \alpha$ and $ \beta$ are quaternionic-valued functions such that
\begin{equation}
\label{co1}
\alpha(u,v)=\alpha(u,-v), \qquad \beta(u,v)=- \beta(u,-v) \quad \hbox{for all} \quad (u,v) \in \mathcal{U}.
\end{equation}
and if the functions $ \alpha$ and $ \beta$ satisfy the Cauchy-Riemann system
$$ \partial_u \alpha(u,v)- \partial_v \beta(u,v)=0, \qquad \hbox{and} \qquad \partial_v \alpha(u,v)+\partial_u \beta(u,v)=0.$$
\end{definition}
The class of left (resp. right) slice hyperholomorphic functions on $U$ is denoted by $ \mathcal{SH}_L(U)$ (resp. $\mathcal{SH}_R(U)$). If the functions $ \alpha$ and $ \beta$ are real-valued, then we are dealing with the subset of intrinsic slice hyperholomorphic functions. This set of functions is denoted by $ \mathcal{N}(U)$.
\\ By means of the left (resp. right) slice hyperholomorphic Cauchy kernel defined by
$$ S_{L}^{-1}(s.q)=(s- \bar{q})(s^2-2q_0s+|q|^2)^{-1}, \qquad \left(\hbox{resp.} \quad S^{-1}_R(s,q)=(s^2-2q_0s+|q|^2)^{-1}(s- \bar{q})\right),$$
for $s\not\in [q]$,
it is possible to give slice hyperholomorphic functions an integral representation. First we recall the definition of slice Cauchy domains that will be often used in the integral representations of functions.
\begin{definition}[Slice Cauchy domain]
An axially symmetric open set $U\subset \hh$ is called a slice Cauchy domain, if $U\cap\cc_J$ is a Cauchy domain in $\cc_J$ for any $J\in\mathbb{S}$. More precisely, $U$ is a slice Cauchy domain if for any $J\in\mathbb{S}$ the boundary ${\partial( U\cap\cc_J)}$ of $U\cap\cc_J$ is the union a finite number of non-intersecting piecewise continuously differentiable Jordan curves in $\cc_{J}$.
\end{definition}
\begin{theorem}[The Cauchy formulas of slice hyperholomorphic functions]
\label{Cauchy}
Let $U \subset \mathbb{H}$ be a bounded slice Cauchy domain, let $J \in \mathbb{S}$ and set $ds_J=ds(-J)$. If a function $f$ is left (reps. right) slice hyperholomorphic on a set that contains $ \bar{U}$, then for any $q \in U$ we have
$$ f(q)= \frac{1}{2 \pi} \int_{\partial(U \cap \mathbb{C}_J)} S^{-1}_L(s,q) ds_Jf(s), \qquad \left( {\rm resp.} \ f(q)= \frac{1}{2 \pi} \int_{\partial(U \cap \mathbb{C}_J)} f(s)ds_JS^{-1}_R(s,q)  \right).$$
\end{theorem}
Similar formulas holds for unbounded domains.
Another class of hyperholomorphic functions   that we will use are the so-called axially monogenic functions, see \cite{DSS}.
\begin{definition}[Axially monogenic functions]
Let $U \subset \mathbb{R}^4$ be an axially symmetric slice domain. A function $f: U \to \mathbb{H}$ of class $\mathcal{C}^1$ is said to be (left) axially monogenic if it is monogenic, i.e.
$$
 \mathcal{D}f(q)= \partial_{q_0}f(q)+ \sum_{i=1}^3 e_i \partial_{q_i} f(q)=0,\ \ \ \ q\in U
$$
and it has the form
\begin{equation}
\label{axial}
f(q)= A(q_0, |\underline{q}|)+ \underline{\omega} B(q_0, |\underline{q}|), \qquad \underline{\omega}:= \frac{\underline{q}}{|\underline{q}|},
\end{equation}
where the functions $A$ and $B$ satisfy the even-odd conditions \eqref{co1}. We denote by $ \mathcal{AM}(U)$ the set of (left) axially monogenic functions.
A similar definition can be given for right axially monogenic functions and we use the notation
$$
 f(q)\mathcal{D}= \partial_{q_0}f(q)+ \sum_{i=1}^3  \partial_{q_i} f(q)e_i=0,\ \ \ \ q\in U.
$$
\end{definition}

The classes of slice hyperholomorphic and axially monogenic functions are connected by the well-known Fueter mapping theorem, see \cite{Fueter}.

\begin{theorem}[Fueter mapping theorem]
\label{FT}
Let $f_0(z)= \alpha(u,v)+i \beta(u,v)$ be a holomorphic function defined in a domain (open and connected) $D$ in the upper-half complex plane and let
$$ \Omega_D=\{q=q_0+ \underline{q} \, | \, (q_0, |\underline{q}|) \in D\}$$
be the open set induced by $D$ in $ \mathbb{H}$. Then the slice operator defined by
$$ f(q)= T_{F1}(f_0):= \alpha(q_0, |\underline{q}|)+ \frac{\underline{q}}{|\underline{q}|} \beta(q_0, |\underline{q}|)$$
maps the set of holomorphic functions in the set of intrinsic slice hyperholomorphic functions.
\\ Furthermore, the function
$$
\breve{f}(q):= T_{F2} \left(\alpha(q_0, |\underline{q}|)+ \frac{\underline{q}}{|\underline{q}|} \beta(q_0, |\underline{q}|)\right),
$$
where $ T_{F2}:= \Delta$ is the Laplace operator in four real variables, is an axially monogenic function.
\end{theorem}
\begin{remark}
In the Fueter mapping theorem
 the restriction that the holomorphic function has to be defined in the upper-half complex plane can be removed when we impose the even-odd conditions (\ref{co1}).
\end{remark}
\begin{remark}
We can summarize the Fueter construction by the following diagram
\begin{equation}
\label{diagintro}
\begin{CD}
	\textcolor{black}{\mathcal{O}(D)}  @>T_{F1}>> \textcolor{black}{\mathcal{SH}(\Omega_D)}  @>\ \   T_{F2}=\Delta >>\textcolor{black}{\mathcal{AM}(\Omega_D)},
\end{CD}
\end{equation}
where $ \mathcal{O}(D)$ is the set of holomorphic functions defined on $D$.
\end{remark}
\begin{remark}
In the Clifford algebra setting the operator $T_{F2}$ is equal to $ \Delta_{n+1}^{\frac{n-1}{2}}$, where $\Delta_{n+1}$ is the Laplace operator in $n+1$ variables. If $n$ is odd, then the $T_{F2}$ is a pointwise differential operator, see \cite{ColSabStrupSce, Sce}.
If $n$ is even we are dealing with the fractional powers of the Laplace operator, see \cite{TaoQian1}.
\end{remark}

The Cauchy formula of slice hyperholomorphic function allows to prove the Fueter theorem in integral form, see \cite{CSS}.
\begin{theorem}[Fueter theorem in integral form]
\label{Fintform}
Let $U \subset \mathbb{H}$ be a slice Cauchy domain, let $J \in \mathbb{S}$ and set $ds_J=ds(-J)$.
If $f$ is a left (reps. right) slice hyperholomorphic function on a set $W$, such that $ \bar{U} \subset W$, then the axially monogenic function $ \breve{f}(q)= \Delta f(q)$ admits the integral representation
$$ \breve{f}(q)= \frac{1}{2 \pi} \int_{\partial(U \cap \mathbb{C}_J)} F_L(s,q) ds_J f(s),
\qquad \left({\rm resp.} \ \breve{f}(q)= \frac{1}{2 \pi} \int_{\partial(U \cap \mathbb{C}_J)}  f(s)ds_J F_R(s,q) \right).$$
\end{theorem}
The kernels of the previous integral transform are given by
$$ F_L(s,q):= \Delta S^{-1}_L(s,q)=-4(s- \bar{q})\mathcal{Q}_{c,s}(q)^{-2}, \ \
({\rm resp.} \ F_R(s,q):=\Delta S^{-1}_R(s,q)=-4\mathcal{Q}_{c,s}(q)^{-2}(s- \bar{x})), $$
where $\mathcal{Q}_{c,s}(q):=s^2-2q_0s+|q|^2$ is the pseudo-Cauchy kernel.
\\The kernel $F_L(s,q)$ (resp. $F_R(s,q)$) is  right slice hyperholomorphic in the variable $s$ (resp. left slice hyperholomorphic) and is left axially monogenic in the variable $q$ (resp. right axially monogenic).

\medskip

It is possible to factorize the Laplace operator appearing in the  diagram \eqref{diagintro}
in the following two different ways. Precisely,
if $\mathcal{D}:=\partial_{q_0}+ \sum_{i=1}^3 e_i \partial_{q_i}$
  and  $ \mathcal{\overline{D}}:=\partial_{q_0}- \sum_{i=1}^3 e_i \partial_{q_i}$  are the Cauchy-Fueter operator and its conjugate, respectively, then
  we have
  $$
  \Delta=\mathcal{D}\overline{\mathcal{D}}=\overline{\mathcal{D}}\mathcal{D}.
  $$
Even though the operators $\mathcal{D}$ and $\overline{\mathcal{D}}$ commute
  the factorizations $\mathcal{D}\overline{\mathcal{D}}$ and $\overline{\mathcal{D}}\mathcal{D}$
  give rise to different functions spaces according to the order
   $\mathcal{D}$ and $\overline{\mathcal{D}}$ applied to a slice hyperholomorphic function.
 In fact, if we first apply operator $\mathcal{D}$ to functions in $\mathcal{SH}(\Omega_D)$ we have:
 $$
\begin{CD}
	\textcolor{black}{\mathcal{O}(D)}  @>T_{F1}>> \textcolor{black}{\mathcal{SH}(\Omega_D)} @>\mathcal{D}>> \textcolor{black}{\mathcal{AH}(\Omega_D)}  @>\ \    \mathcal{\overline{D}} >>\textcolor{black}{\mathcal{AM}(\Omega_D)},
\end{CD}
$$
where $ \mathcal{AH}(\Omega_D)$ is the class of axially harmonic function, i.e., functions of the form \eqref{axial} and in the kernel of the Laplace operator in four real variables.
In the case we apply $\overline{\mathcal{D}}$ to functions in $\mathcal{SH}(\Omega_D)$ we obtain
$$
\begin{CD}
	\textcolor{black}{\mathcal{O}(D)}  @>T_{F1}>> \textcolor{black}{\mathcal{SH}(\Omega_D)} @>\mathcal{\overline{D}}>> \textcolor{black}{\mathcal{AP}_2(\Omega_D)}  @>\ \  \mathcal{D} >>\textcolor{black}{\mathcal{AM}(\Omega_D)},
\end{CD}
$$
where $ \mathcal{AP}_2(\Omega_D)$ is the class of axially polyanalytic functions of order 2, i.e., functions of the form \eqref{axial} and in the kernel of the operator $ \mathcal{D}^2$.
The above cases lead us to the following definition.
\begin{definition}[Fine structures of the spectral theory on the $S$-spectrum]\label{DEFFINE} We will call \emph{fine structures of the spectral theory on the $S$-spectrum} the set of functions spaces and the associated functional calculi induced by a factorization of the operator $T_{F2}$, in the Fueter extension theorem.
\end{definition}
In \cite{CDPS1, Polyf1} we give an integral representation for axially harmonic functions and axially polyanalytic functions of order 2,
we now recall such integral representations.
Let $ W \subset \mathbb{H}$ be an open set. Let $U$ be a slice Cauchy domain such that $ \bar{U} \subset W$.
For $J \in \mathbb{S}$ and $ds_J=ds(-J)$ we have that:
\newline
$\bullet$ If the function $f$ is left (resp. right) slice hyperholomorphic in $W$, then the function $ \tilde{f}(q)= \mathcal{D}f(q)$ (resp. $ f(q) \mathcal{D}$) is harmonic and it admits the following integral representation for $q \in U$
\begin{equation}
\label{inteharmo}
\tilde{f}(q)= \frac{1}{ 2 \pi} \int_{\partial(U \cap \mathbb{C}_J)} \mathcal{D} S^{-1}_L(s,q)  ds_J f(s) \quad \left(\hbox{resp.} \,  \tilde{f}(q)= \frac{1}{2 \pi} \int_{\partial(U \cap \mathbb{C}_J)}  f(s) ds_J  S^{-1}_R(s,q)\mathcal{D}  \right).
\end{equation}
$\bullet$ If the function $f$ is left (resp. right) slice hyperholomorphic in $W$, then the function $ \breve{f}^{0}(q)= \mathcal{\overline{D}}f(q)$ (resp. $ f(q) \mathcal{\overline{D}}$) is polyanalytic of order 2 and it admits the following integral representation for $q \in U$
\begin{equation}
\label{intepoly}
\breve{f}^{0}(q)= \frac{1}{2 \pi} \int_{\partial(U \cap \mathbb{C}_J)} \mathcal{\overline{D}} S^{-1}_L(s,q) ds_J f(s) \quad \left(\hbox{resp.} \, \breve{f}^{0}(q)= \int_{\partial(U \cap \mathbb{C}_J)} f(s) ds_J  S^{-1}_R(s,q)\mathcal{\overline{D}} \right).
\end{equation}
Our main results are an extension of  the Riesz-Dunford functional calculus, see \cite{RD} for the classes of functions of the quaternionic fine structures.
Precisely, let $f$ be a holomorphic function defined on an open set containing $ \overline{\Omega}$,
where $ \Omega \subset \mathbb{C}$ is an open set that contains the spectrum of the bounded operator $B$ on a complex Banach space $X_\mathbb{C}$.
 The holomorphic functional calculus $f(B)$ is defined as
$$ f(B)= \frac{1}{2 \pi i} \int_{\partial \Omega} (\lambda \mathcal{I}-B)^{-1}f(\lambda) d \lambda.$$
In order to define the functional calculus for an unbounded operator $A$, we assume that the function $f$ is holomorphic at infinity and on the spectrum $\sigma(A)$ of the closed linear operator $A$ whose domain is contained in $X_\mathbb{C}$.
\\Let us assume that the resolvent $\rho(A)$ is nonempty. We set
$\Phi(\lambda):= (\lambda- \alpha)^{-1}$, $ \Phi(\infty)=0$ and $ \Phi(\alpha)=\infty$, for $ \alpha \in \rho(A)$. We define the functional calculus for unbounded operators as $f(A):=\phi(B)$ where $\phi(\mu)= f(\Phi^{-1}(\mu))$ and $B:= (A- \alpha \mathcal{I})^{-1}. $ For this functional calculus we have the following integral representation
\begin{equation}
\label{Ries}
f(A)= f(\infty) \mathcal{I}+ \frac{1}{2 \pi i} \int_{\Gamma} (\lambda \mathcal{I}-A)^{-1} f(\lambda) d \lambda,
\end{equation}
where $\Gamma$ consists of a finite numbers of Jordan arcs,
that surround the spectrum $\sigma(A)$ and the point at infinity.
We recall that the function $f$ has to be holomorphic in an open set that contains $ \Gamma$ and its interior.
Moreover, formula \eqref{Ries} implies that $f(A)$ is independent from the parameter $ \alpha$.
\newline
\newline
The harmonic and polyanalytic functional calculi for bounded quaternionic operators with
commuting components were introduced
in \cite{CDPS1, Polyf1} where operators of the form $T=T_0+e_1T_1+e_2T_2+e_3T_3$ are considered. The operator $ \bar{T}=T_0-e_1T_1-e_2T_2-e_3T_3$ is called the conjugate  of $T$.
The harmonic resolvent operator is defined as
$$ \mathcal{Q}_{c,s}(T)^{-1}:= (s^{2}-(T+ \bar{T})s+T \bar{T})^{-1}, \ \ \ s\in \rho_S(T),$$
and it tuned out to be the well-known commutative pseudo $S$-resolvent operator.
The polyanalytic (left) resolvent operator is defined as
$$ P_2^L(s,T)=-F_L(s,T)s+T_{0}F_L(s,T), \quad \hbox{where} \quad F_L(s,T)=-4(s- \bar{T}) \mathcal{Q}_{c,s}(T)^{-2}.$$
Similarly we defined the polyanalytic right resolvent operator.
In \cite{CDPS1, Polyf1} we proved that the harmonic functional calculus (also called $Q$-functional) for $f \in \mathcal{SH}_L(U)$, given by
$$ \tilde{f}(T)=- \frac{1}{\pi} \int_{\partial(U \cap \mathbb{C}_J)} \mathcal{Q}_{c,s}(T)^{-1} ds_J f(s),$$
and the $P_2$-functional calculus, defined as
$$ \breve{f}^{\circ}(T):= \frac{1}{2 \pi} \int_{\partial(U \cap \mathbb{C}_J)} P_2^L(s,T) ds_J f(s),$$
are well defined since the integrals do not depend on the open set $U$ nor on $J \in \mathbb{S}$. Here $U$ denotes a suitable open set that has smooth boundary and contains the $S$-spectrum of $T$.
\newline
\newline
The goal of this paper is to define the  harmonic functional calculus and the $P_2$-functional calculus for commuting unbounded operators
in the spirit of the holomorphic functional calculus for closed operators mentioned above.
Let us consider $f$ to be a suitable slice hyperholomorphic function and let
$\tilde{\psi}$ be defined as:
$$ \tilde{\psi}(q)= \mathcal{D}(\phi(q)),$$
where $\phi(q):=(f \circ \phi_{\alpha}^{-1})(q)$, $q:= \phi_{\alpha}(s)= (s- \alpha)^{-1}$, $\phi_{\alpha}(\infty)=0$, $\phi_{\alpha}(\alpha)=\infty$. We set $A:=(T- \alpha \mathcal{I})^{-1}$ for $ \alpha \in \rho_S \cap \mathbb{R} \neq \emptyset$. The $Q$-functional calculus for  closed quaternionic operators is defined as
$$ \tilde{f}(T):=(A \bar{A}) \tilde{\psi}(A).$$
By using the same notations we define
$$ \breve{\psi}^0(q):= \mathcal{\overline{D}}(q^2\phi(q)),$$
and, by means of this function, we define the $P_2$-functional calculus for closed operators
$$ \breve{f}^0(T):= \breve{\psi}^0(A).$$
Moreover, we show integral representations for $ \tilde{f}(T)$ and $ \breve{f}^0(T)$, which are the counter part of formula \eqref{Ries} for the Riesz-Dunford functional calculus.
\\This fact is not necessarily expected because the $Q$-functional and $P_2$-functional calculi are defined through integral transforms, whereas the Riesz-Dunford functional calculus is based on the Cauchy formula of holomorphic functions.
\newline
\newline
\emph{Outline of the paper.} The paper consists of five sections, the first one being this introduction. In Section 2 we briefly revise the $S$-functional calculus and the $F$-functional calculus, both for bounded and unbounded operators with commuting components with some new insight to this calculus.
 In Section 3 we establish and prove the main properties for the harmonic functional calculus for closed operators.
 In Section 4 we define the $P_2$-functional calculus for closed operators and we prove its main properties.  Section 5 contains some  concluding remarks.

\section{Preliminaries results on quaternionic operators}
In this section we recall some results on the $S$-functional calculus and $F$-functional calculus for bounded and unbounded operators, see for more details \cite{CGK},
 in order to give a complete picture of the fine structures on the $S$-spectrum.

\medskip
{\bf Bounded quaternionic linear operators.}
Let $X$ be a two sided quaternionic Banach module of the form $X= X_{\mathbb{R}} \otimes \mathbb{H}$, where $X_{\mathbb{R}}$ is a real Banach space. We denote by $ \mathcal{B}(X)$ the Banach space of all bounded right linear operators acting on $X$.
In order to define the $S$-resolvent set and the $S$-spectrum  we define
the operator
$$ \mathcal{Q}_s(T):= T^{2}-2s_0T+|s|^2 \mathcal{I}$$
for $T \in \mathcal{B}(X)$.

\begin{definition}
Let $T \in \mathcal{B}(X)$. We define the $S$-resolvent set $\rho_S(T)$ of $T$ as
$$ \rho_S(T):= \{s \in \mathbb{H}\, : \, \mathcal{Q}_s(T)^{-1} \in \mathcal{B}(X)\}.$$
The $S$-spectrum $\sigma_S(T)$ of $T$ is defined as
$$ \sigma_S(T):= \mathbb{H} \setminus \rho_S(T).$$
The operator
$$
\mathcal{Q}_s(T)^{-1}= (T^{2}-2s_0T+|s|^2 \mathcal{I})^{-1}, \ \ {\rm for} \ \ s\in \rho_S(T)
$$
is called the pseudo $S$-resolvent operator of $T$ at $s$.
\end{definition}

Due to the lack of commutativity there exist two different resolvent operators.

\begin{definition}
Let $T \in \mathcal{B}(X)$ and $s \in \rho_S(T)$. Then the left (resp. right) $S$-resolvent operator is defined as
$$ S^{-1}_L(s,T)=- \mathcal{Q}_s^{-1}(T)(T- \bar{s} \mathcal{I}), \qquad \left( \hbox{resp.} \quad S^{-1}_R(s,T):=-(T- \bar{s} \mathcal{I}) \mathcal{Q}_s(T)^{-1} \right).$$
\end{definition}

We denote by $ \mathcal{SH}_L(\sigma_S(T))$, $ \mathcal{SH}_R(\sigma_S(T))$ and $ \mathcal{N}(\sigma_S(T))$ the sets of left, right and intrinsic slice hyperholomorphic functions, respectively, with $ \sigma_{S}(T) \subset \hbox{dom}(f)$.
For these classes of functions we introduce the $S$-functional calculus for bounded operators.
\begin{definition}[$S$-functional calculus for bounded operators]
\label{Sfun}
Let $U$ be a bounded slice Cauchy domain that contains $\sigma_S(T)$ and suppose that
$ \overline{U}$ is contained in the domain of $f$. Set $ds_J=(-J)ds$ for $J \in \mathbb{S}$ and let $f \in \mathcal{SH}_L(\sigma_S(T))$ (resp. $\mathcal{SH}_R(\sigma_S(T))$) we define
\begin{equation}
\label{S1}
f(T):= \frac{1}{ 2 \pi} \int_{\partial(U \cap \mathbb{C}_J)} S^{-1}_L(s,T) ds_J f(s), \quad \left({\rm resp}. \, \, f(T):= \frac{1}{ 2 \pi} \int_{\partial(U \cap \mathbb{C}_J)} f(s) ds_JS^{-1}_R(s,T)\right).
\end{equation}
\end{definition}
The $S$-functional calculus for bounded operators is well posed, i.e., the integrals in \eqref{S1} depend neither on $U$ nor on the imaginary unit $J \in \mathbb{S}$. Standard references for the $S$-functional calculus are the books \cite{FJBOOK,CGK, ColomboSabadiniStruppa2011}.

Now, we recall the $F$-functional calculus, which was introduced for the first time in \cite{CSS} and further developed in \cite{CG, CDS,CDS1}. This is monogenic functional calculus in the same spirit of McIntosh and collaborators, see \cite{J, JM}, but it is based on the theory of the $S$-spectrum.
The $F$-functional calculus arises by the Fueter theorem in integral form (see Theorem \ref{Fintform})  by formally replacing the quaternion $q$ with a suitable quaternionic operator $T$.

For this functional calculus we will consider bounded operators of the form $T=T_0+e_1T_1+T_2e_2+T_3e_3$ with commuting components $T_i$ acting on a real vector space $X_{\mathbb{R}}$, i.e. $T_i \in \mathcal{B}(X_{\mathbb{R}})$ for $i=0,1,2,3$. The subset of $ \mathcal{B}(X)$ given by operators $T$ with commuting components $T_i$ will be denoted by $ \mathcal{BC}(X)$, for this operators we define the $S$-spectrum in a different, but equivalent way.

\begin{definition}
Let $T \in \mathcal{BC}(X)$. For $s \in \mathbb{H}$ we set
$$ \mathcal{Q}_{c,s}(T):= s^{2} \mathcal{I}-s (T+ \bar{T})+ T \bar{T},$$
where $ \bar{T}:=T_0-T_1e_1-T_2e_2-T_3e_3$. We define the $F$-resolvent set as
$$ \rho_F(T):= \{s \in \mathbb{H}\,: \, \mathcal{Q}_{c,s}(T)^{-1} \in \mathcal{B}(X)\},$$
and the $F$-spectrum of $T$ as
$$ \sigma_F(T)= \mathbb{H} \setminus \rho_F(T).$$
\end{definition}
In \cite{CS} it is proved that the $F$-spectrum is the commutative version of the $S$-spectrum, i.e. for $T \in \mathcal{BC}(X)$ we have
$$ \sigma_F(T)= \sigma_S(T).$$
The operator $\mathcal{Q}_{c,s}(T)^{-1}$, for $s\in\rho_F(T)$,  is called the commutative pseudo $S$-resolvent operator of $T$ at $s$.
In \cite{CS} a commutative version of the $S$-functional calculus for bounded operators is established and in this case we have to deal with the following resolvent operators.

\begin{definition}
Let $T \in \mathcal{BC}(X)$ and $s \in \rho_S(T)$. We define the left (resp. right) commutative $S$-resolvent operator as
$$ S^{-1}_L(s,T)=( s\mathcal{I}- \bar{T}) \mathcal{Q}_{c,s}(T)^{-1}, \quad \left({\rm resp}. \quad S^{-1}_R(s,T)= \mathcal{Q}_{c,s}(T)^{-1}(s \mathcal{I}- \bar{T})\right).$$
\end{definition}

\begin{remark}
For the sake of simplicity we denote the commutative version of the $S$-resolvent operators with the same symbol adopted for the noncommutative ones since no ambiguity will arise in the sequel.
\end{remark}
By means of the above resolvent operators it is possible to define the commutative $S$-functional calculus as in Definition \ref{Sfun}.
Now, we have all the tools to give the definition for the $F$-functional calculus.

\begin{definition}[$F$-resolvent operators]
Let $T \in \mathcal{BC}(X)$. We define the left (resp. right) $F$-resolvent operator for $s \in \rho_S(T),$ as
\begin{equation}
\label{star1}
F_{L}(s,T)=-4 ( s\mathcal{I}- \bar{T}) \mathcal{Q}_{c,s}(T)^{-2},  \quad \left({\rm resp}. \quad F_R(s,T)=-4 \mathcal{Q}_{c,s}(T)^{-2}(s \mathcal{I}- \bar{T}) \right).
\end{equation}
\end{definition}

 \begin{definition}[$F$-functional calculus for bounded operators]
Let $U$ be a slice Cauchy domain that contains $\sigma_S(T)$ and assume that
$\overline{U}$ is contained in the domain of $f$. Let $T=T_1e_1+T_2e_2+T_3e_3 \in \mathcal{BC}(X)$, assume that the operators $T_{\ell}$, $ \ell=1,2,3$ have real spectrum and set $ds_J=ds(-J)$, where $J \in \mathbb{S}$. For $ \breve{f}= \Delta f(q)$ with $f \in \mathcal{SH}_L(\sigma_S(T))$ (resp. $f \in \mathcal{SH}_R(\sigma_S(T)))$ we define
\begin{equation}
\label{f1inte}
\breve{f}(T):= \frac{1}{ 2\pi} \int_{\partial(U \cap \mathbb{C}_J)} F_L(s,T) ds_j f(s) \quad \left({\rm resp}. \quad \breve{f}(T):= \frac{1}{ 2\pi} \int_{\partial(U \cap \mathbb{C}_J)} f(s) ds_j F_R(s,T)\right).
\end{equation}
 \end{definition}
 As well as the $S$-functional calculus the $F$-functional calculus is well defined, since the integrals in \eqref{f1inte} depend neither on $U$ and nor on the imaginary unit $J \in \mathbb{S}.$

 \begin{remark}\label{RMKDELKER}
 Since  $\breve{f}= \Delta f(q)$  we observe that $\Delta$ has a kernel and the independence on the kernel
 is discussed in Remark 7.1.13 in \cite{CGK}.
 The independence of the $F$-functional calculus from the ${\rm ker}(\Delta)$ is obtained
  by an application  of the monogenic functional calculus of McIntosh
  that requires that operators $T_{\ell}$, $ \ell=0,1,2,3$ have real spectrum
  and one of the $T_{\ell}$ has to be the zero operator. This requirement is enforced
  in order not to disconnect the monogenic resolvent set and it is usually take $T_0$ to be zero.
 \end{remark}

% All the construction exposed so far can be visualized in the following diagram
%\begin{equation}
%\label{Ffund}
%	\begin{CD}
%		{\mathcal{SH}(\Omega_D)} @.  {\mathcal{AM}(\Omega_D)} \\   @V  VV
%		@.
%		\\
%		{{\rm  Slice\ Cauchy \ Formula}}  @> \Delta>> {{\rm Fueter\ theorem \ in \  integral\  form}}
%		\\
%		@V VV    @V VV
%		\\
%		{S-{\rm Functional \ calculus}} @. {F-{\rm Functional \ calculus}}
%	\end{CD}
%\end{equation}

\medskip
{\bf Unbounded quaternionic linear operators.}
Now we recall the main notions for the $S$-functional calculus and $F$-functional calculus for unbounded operators, introduced in \cite{CS, CSF}. These functional calculi are defined for a peculiar class of operators, that are defined as follows.

\begin{definition}
Let $X$ be a two-sided Banach space. A right linear operator $T: \hbox{dom}(T) \subset X \to X$ is called closed if its graph is closed in the Cartesian product $X \times X$.
\end{definition}

\begin{definition}
Let $T_{\ell}: \hbox{dom}(T_\ell) \subset X \to X$ be a linear closed operators for $ \ell=0,...,3$ such that $T_{\ell}T_k=T_kT_{\ell}$ on $ \hbox{dom}(T_{\ell} T_k) \cap \hbox{dom}(T_k T_{\ell})$ for $ \ell$, $k=0,...,3$. Then $ \hbox{dom}(T)= \cap_{\ell=0}^3 \hbox{dom}(T_{\ell})$ is the domain of the quaternionic right linear operators $T=T_0+ \sum_{\ell=1}^3 e_{\ell} T_{\ell}: \, \hbox{dom}(T) \subset X \to X,$ and $\bar{T}=T_0- \sum_{\ell=1}^3 e_{\ell} T_{\ell}: \, \hbox{dom}(T) \subset X \to X$, since $\hbox{dom}(T)=\hbox{dom}(\bar{T})$. We denote this set of closed right linear operators with commuting components by $ \mathcal{KC}(X)$.
\end{definition}
Now, we give the notion of $S$-spectrum in this setting.
\begin{definition}
Let $T \in \mathcal{KC}(X)$. We denote by $\rho_S(T)$ the $S$-resolvent set of $T$ defined as
$$ \rho_S(T):= \{s \in \mathbb{H}\, : \, \mathcal{Q}_{c,s}^{-1}(T) \in \mathcal{B}(X)\},$$
with
$$
\mathcal{Q}_{c,s}(T):=s^{2} \mathcal{I}-2sT_0+T \bar{T}, \ \ {\rm where}\ \
\mathcal{Q}_{c,s}(T): {\rm dom}(T\bar{T})\to X.
$$
 We define the $S$-spectrum $\sigma_S(T)$ of $T$ as
$$ \sigma_S(T)= \mathbb{H} \setminus \rho_S(T).$$
The extended $S$-spectrum is defined as
$$ \bar{\sigma}_S(T):= \sigma_S(T) \cup \{\infty\}.$$
\end{definition}

\begin{remark}
In the sequel we will always assume that the $S$- resolvent is nonempty.
\end{remark}
\begin{definition}
Let $T \in \mathcal{KC}(X)$. For $s \in \rho_S(T)$ we define
$$
 \mathcal{Q}_{c,s}(T)^{-1}:= (s^{2}-s(T+ \bar{T})+T \bar{T})^{-1},
$$
the commutative pseudo $S$-resolvent operator of $T$ at $s$,
where $\mathcal{Q}_{c,s}(T)^{-1}$
maps $X$ to $\hbox{dom}(T \bar{T})$ for $s\in \rho_S(T)$.
\end{definition}

\begin{definition}[The (commutative) $S$-resolvent operators in the unbounded case]\label{S-reslUNCOOM}
Let $T \in \mathcal{KC}(X)$. For $s \in \rho_S(T)$, we define the left (resp. right) $S$-resolvent operator of $T$ at $s$ as
$$S^{-1}_L(s,T)=(s \mathcal{I}- \bar{T}) \mathcal{Q}_{c,s}(T)^{-1}\quad \left(\hbox{reps.} \quad S^{-1}_R(s,T)=\mathcal{Q}_{c,s}(T)^{-1}s - \bar{T}\mathcal{Q}_{c,s}(T)^{-1} \right).$$
\end{definition}
Observe that because of how $S^{-1}_L(s,T)$ and $S^{-1}_R(s,T)$ are defined
they map $X$ to $\hbox{dom}(T)$.

\begin{definition}
Let $T \in \mathcal{KC}(X)$. A function $f$ is left slice hyperholomorphic in $ \overline{\sigma}_S(T)$ if and only if $f$ is left slice hyperholomorphic with $\sigma_S(T)\subset \hbox{dom}(f)$ and if furthermore, $ \mathbb{H} \setminus B_r(0)\subset \hbox{dom}(f)$ for some $r >0$ and $f( \infty)= \lim_{q \to \infty} f(q)$ exists.
We denote this set of functions as $\mathcal{SH}_L(\overline{\sigma}_S(T))$.
\end{definition}
Similarly we can define the functions in $ \mathcal{SH}_R(\overline{\sigma}_S(T))$ and $\mathcal{N}(\overline{\sigma}_S(T))$.
\begin{definition}\label{phialp}
	\label{homo}
	For $ \alpha \in \rho_S(T) \cap \mathbb{R}\not=\emptyset$, we define the homomorphism
$ \phi_{\alpha}: \overline{\mathbb{H}} \to \overline{\mathbb{H}}$ as
	\begin{equation}
		\label{1first}
		p=\phi_{\alpha}(s)=(s- \alpha)^{-1}, \qquad \phi_{\alpha}(\alpha)=\infty, \qquad\phi_{\alpha}(\infty)=0.
	\end{equation}
\end{definition}
\begin{proposition}\label{simaAT}
Let $A:=(T- \alpha \mathcal{I})^{-1}$, for $ \alpha \in \rho_S(T) \cap \mathbb{R}$ and $ \phi_{\alpha}$ defined as above, then the mapping
$
f \mapsto f \circ \phi_{\alpha}^{-1}
$
 determines one-to-one correspondences between $\mathcal{SH}_L(\overline{\sigma}_S(T))$ (resp. $\mathcal{SH}_R(\overline{\sigma}_S(T))$ and  $\mathcal{SH}_L(\sigma_S(A))$ (resp. $\mathcal{SH}_L(\overline{\sigma}_S(T))$, and between $ \mathcal{N}(\overline{\sigma}_S(T))$ and $ \mathcal{N}(\sigma_S(A))$. Precisely we have

\[
\begin{split}
&\mathcal{SH}_L(\sigma_S(A))= \{f \,\circ \, \phi_{\alpha}^{-1}\, : \, f \in \mathcal{SH}_L(\overline{\sigma}_S(T))\},
\\
&
\mathcal{SH}_R(\sigma_S(A))= \{f \circ \phi_{\alpha}^{-1}\, : \, f \in \mathcal{SH}_R(\overline{\sigma}_S(T))\},
\\
&
\mathcal{N}(\sigma_S(A))= \{f \circ \phi_{\alpha}^{-1}\, : \, f \in \mathcal{N}_L(\overline{\sigma}_S(T))\}.
\end{split}
\]
\end{proposition}
\begin{proof}
It is \cite[Cor. 5.2.4]{CGK}.
\end{proof}
\begin{definition}[$S$-functional for unbounded operators]
Let $T \in \mathcal{KC}(X)$.
   We assume that $\rho_S(T) \cap \mathbb{R} \neq \emptyset$ and suppose that $f \in \mathcal{SH}_L(\overline{\sigma}_S(T))$. Let us consider
	$\phi(p)=f(\phi_{\alpha}^{-1}(p))$
	and the operator
	$$ A:=(T-\alpha \mathcal{I})^{-1}, \qquad \alpha \in \rho_S(T) \cap \mathbb{R}.$$
	We define
	\begin{equation}
		\label{Sfun1}
		f(T):=\phi(A).
	\end{equation}
\end{definition}
We observe that the $S$-functional calculus for commuting unbounded operators is well defined. This is based on the fact that $f(T)=\phi((T-\alpha \mathcal{I})^{-1})$ does not depend on $ \alpha$, thanks to the following theorem.

\begin{theorem}
Let $T \in \mathcal{KC}(X)$ and
let $\rho_S(T) \cap \mathbb{R} \neq\emptyset$. Assume that $f \in \mathcal{SH}_L(\overline{\sigma}_S(T))$ and $f(T)$ is the operator defined in \eqref{Sfun1}. Then
$$ f(T)=f(\infty) \mathcal{I}+ \frac{1}{2 \pi} \int_{\partial(U \cap \mathbb{C}_J)}S^{-1}_L(s,T) ds_J f(s),$$
for every unbounded slice Cauchy domain with $\sigma_S(T) \subset U$ and $ \overline{U} \subset \hbox{dom}(f)$ and every imaginary unit $J \in \mathbb{S}$. Similarly, if $f \in \mathcal{SH}_R(\overline{\sigma}_S(T))$ and $f(T)$ the operator defined in \eqref{Sfun1}, then
$$ f(T)=f(\infty) \mathcal{I}+ \frac{1}{2 \pi} \int_{\partial(U \cap \mathbb{C}_J)} f(s) ds_J S^{-1}_R(s,T),$$
for every unbounded slice Cauchy domain with $\sigma_S(T) \subset U$ and $ \overline{U} \subset \hbox{dom}(f)$ and every imaginary unit $J \in \mathbb{S}$.
\end{theorem}

To prove the above theorem it is crucial to recall the relation between the resolvents $S^{-1}_L(s,T)$ and $S^{-1}_L(p,A)$ given by the following result, see \cite[Thm. 6.12]{CS}.

\begin{theorem}
Let $ \alpha \in \rho_S(T) \cap \mathbb{R}$. If $\phi_{\alpha}$ is defined as \eqref{1first}, then $\phi_{\alpha}(\overline{\sigma}_S(T))=\sigma_S(A)$. Moreover, the following identity holds
\begin{equation}
\label{Sunbou}
S^{-1}_L(p,A)=p^{-1} \mathcal{I}-S^{-1}_L(s,T) p^{-2},
\end{equation}
where $S^{-1}_L(p,A)$ is the commutative $S$-resolvent operator in the unbounded case in Definition \ref{S-reslUNCOOM}.

\end{theorem}

Now, we recall few properties of the $S$-functional calculus for unbounded operators, see \cite[Theorem 6.14]{CS}.
\begin{theorem}\label{proper}
Let $T \in \mathcal{KC}(X)$ with $\rho_S(T) \cap \mathbb{R} \neq \emptyset$.
\begin{itemize}
\item If $f$, $g \in \mathcal{SH}_L(\overline{\sigma}_S(T))$ and $a \in \mathbb{H}$,  then
$$ (fa+g)(T)=f(T)a+g(T),$$
\item If $f \in \mathcal{N}(\overline{\sigma}_S(T))$ and $g \in \mathcal{SH}_L(\overline{\sigma}_S(T))$ or $f \in \mathcal{SH}_R(\overline{\sigma}_S(T))$ and $g \in \mathcal{N}(\overline{\sigma}_S(T)$, then
$$ (fg)(T)=f(T)g(T).$$
\item If $s$, $q \in \rho_S(T)$ with $s \notin [q]$ then
\begin{eqnarray*}
S^{-1}_R(s,T)S^{-1}_L(q,T)&=&[\left(S^{-1}_R(s,T)-S^{-1}_L(q,T)\right)q\\
&&- \bar{s}\left(S^{-1}_R(s,T)-S^{-1}_L(q,T)\right)](q^2-2s_0q+|s|^2)^{-1}.
\end{eqnarray*}
\end{itemize}
\end{theorem}
The following result puts in relation the commutative pseudo S-resolvent of $T$ with the one of $A$, see \cite[Theorem 8.1.1]{CGK}.
\begin{theorem}
\label{protQAQT}
Let $T\in\mathcal{KC}(X)$ and assume that there exists a point $\alpha\in \rho_{S}(T)\cap \mathbb{R}\not=\emptyset$ and set  $A:=(T-\alpha\id)^{-1}$. For $p =\phi_{\alpha}(s)=(s - \alpha)^{-1}$, we have
\[
\begin{split}
\Q_{c,p}(A)^{-1}& = \left( A\overline{A} \right)^{-1}\Q_{c,s}(T)^{-1}p^{-2}
\end{split}
\]
and
\[
\Q_{c,p}(A)^{-2} = \left( A\overline{A} \right)^{-2}\Q_{c,s}(T)^{-2}p^{-4}.
\]
\end{theorem}

Now, we revise the $F$-functional calculus for unbounded operators.

\begin{definition}\label{UNBFRES}
Let $T \in \mathcal{KC}(X)$. For $s \in \rho_S(T)$. We define the left (resp. right) $F$-resolvent operators as
$$ F_L(s,T)= -4(s \mathcal{I}- \bar{T}) \mathcal{Q}_{c,s}(T)^{- 2}, \quad \left( \hbox{resp.} \quad  F_R(s,T)= -4 \mathcal{Q}_{c,s}(T)^{- 2}s- \bar{T}\mathcal{Q}_{c,s}(T)^{- 2} \right).$$
\end{definition}
The $F$-functional calculus for unbounded operators considered in this paper
will be based on the following relation between
the $F$-resolvent operators $F_L(p,A)$ and $F_{L}(s,T)$.
The difference with the definition in \cite{CSF} is in the powers of the operator $A \bar{A}$.
\begin{theorem}
Let $T\in \mathcal{KC}(X)$, let $ \alpha \in \rho_S(T) \cap \mathbb{R} \neq \emptyset$ and define $A=(T- \alpha \mathcal{I})^{-1}$. For $s \in \rho_S(T)$ and $p=\phi_{\alpha}(s)=(s- \alpha)^{-1}$, we have
\begin{equation}\label{NEWFAP}
(A \bar{A})F_L(p,A) p^4v=-4 p \mathcal{Q}_{c,s}(T)^{-1}v-F_{L}(s,T)v, \ \ \ v\in X.
\end{equation}
\end{theorem}
\begin{proof}
We start by observing that we can write the $F$-resolvent operator $F_L(p,A)$ as
$$ F_L(p,A)=-4 S^{-1}_L(p,A) \mathcal{Q}_{c,p}(A)^{-1}.$$
By using formula \eqref{Sunbou} and Theorem \ref{protQAQT} we get
$$ F_L(p,A)=-4 \left(p^{-1} \mathcal{I}-S^{-1}_L(s,T)p^{-2}\right) (A \bar{A})^{-1} \mathcal{Q}_{c,s}(T)^{-1}p^{-2}.$$
Thanks to the hypothesis on $T$ and consequently on $A$ the operator
$$
A\overline{A}=(\alpha ^2\mathcal{I}-\alpha (T+\overline{T})+T\overline{T})^{-1}: \ X\to {\rm dom}(T\overline{T})
$$
does not contain imaginary units as well as
$$
(A\overline{A})^{-1}=\alpha ^2\mathcal{I}-\alpha (T+\overline{T})+T\overline{T}: \ {\rm dom}(T\overline{T})\to X.
$$
This implies that we can commute $(A\overline{A})^{-1}$ with $p^{-1}$ and its powers since $p=(s- \alpha)^{-1}$.
Moreover, $(A\overline{A})^{-1}$ commutes also with the
commutative version of the $S$-resolvent
$$
 S^{-1}_L(s,T)=(s \mathcal{I}- \bar{T}) \mathcal{Q}_{c,s}(T)^{-1}.
 $$
Regarding the domains of the operators we recall that
 $$
 \mathcal{Q}_{c,s}(T)^{-1}: X\to \hbox{dom}(T \bar{T})
 $$
and
$$
S^{-1}_L(s,T): X\to \hbox{dom}(T)
$$
 so we have
$$
 S^{-1}_L(s,T)p^{-2} (A \bar{A})^{-1} \mathcal{Q}_{c,s}(T)^{-1}p^{-2}v=(A \bar{A})^{-1}S^{-1}_L(s,T)p^{-2}  \mathcal{Q}_{c,s}(T)^{-1}p^{-2}v $$
 for all $v\in X$ and for all  for $s\in \rho_S(T)$.
Observe now that we obtain
\begin{eqnarray*}
F_{L}(p,A)v&=&(A \bar{A})^{-1}[-4 \mathcal{Q}_{c,s}(T)^{-1}p^{-3}+4 S^{-1}_L(s,T) \mathcal{Q}_{c,s}(T)^{-1}p^{-4}]v\\
&=& (A \bar{A})^{-1} \left(-4 \mathcal{Q}_{c,s}(T)^{-1}p^{-3}-F_L(s,T)p^{-4}\right)v
\end{eqnarray*}
and the relation holds for $v\in X$. This fact is justified
by observing that
$$
\mathcal{Q}_{c,s}(T)^{-2}: X\to {\rm dom}((T\bar{T})^2)
$$
where $ {\rm dom}((T\bar{T})^2)={\rm dom}(T^4)$,
so the range of $-4\mathcal{Q}_{c,s}(T)^{-1}p^{-3}-F_L(s,T)p^{-4}$
 is contained in the domain of $(A \bar{A})^{-1}$.
 Moreover, the range of $F_{L}(p,A)$ in contained in the domain of $A \bar{A}$.
 With these observations
we can multiplying on the left-hand side by $A \bar{A}$ and on the right-hand side by $p^{4} \mathcal{I}$ to get the statement.
\end{proof}
Formula (\ref{NEWFAP}) is suitable to define the $F$-functional calculus for unbounded operators.
It is worth noticing that the above formula gives a definition for the $F$-functional calculus, slightly different from the customary definition used in the literature, see \cite{CGK}.

Keeping in mind the definition of $\phi_\alpha$ and its property, see Definition \ref{phialp} and Proposition \ref{simaAT}, we are in the position to define the $F$-functional calculus for unbounded operators.
\begin{definition}[$F$-functional calculus for unbounded operators]
Let $T \in \mathcal{KC}(X)$ and
 assume that the operators $T_{\ell}$, $ \ell=0,1,2,3$ have real spectrum,
  where one of the $T_{\ell}$ is the zero operator.
  Suppose that $\rho_S(T) \cap \mathbb{R} \neq \emptyset$, let $ \alpha \in \rho_S(T) \cap \mathbb{R}$, define $ \phi_{\alpha}$ as in \eqref{1first} and set $A:=(T- \alpha \mathcal{I})^{-1}$. For $f \in \mathcal{SH}_L(\overline{\sigma}_S(T))$ with $f(\alpha)=0$, we consider the functions
$$ \phi(q)= (f \circ \phi_{\alpha}^{-1})(q),$$
$$ \breve{\phi}(q)= \Delta (q^2 \phi(q))$$
and we define the operator $\breve{f}(T)$, for $ \breve{f}=\Delta f$, as
\begin{equation}
\label{Ffun}
\breve{f}(T):=(A \bar{A}) \breve{\phi}(A),
\end{equation}
where $ \breve{\phi}(A)$ is defined by means of the bounded $F$-functional calculus.
A similar definition holds for $f \in \mathcal{SH}_R(\overline{\sigma}_S(T))$ with $f(\alpha)=0$.
\end{definition}

\begin{theorem}
Let $T \in \mathcal{KC}(X)$ and
 assume that the operators $T_{\ell}$, $ \ell=0,1,2,3$ have real spectrum,
  where one of the $T_{\ell}$ is the zero operator.
Suppose that $\rho_S(T) \cap \mathbb{R} \neq \emptyset$, let $ \alpha \in \rho_S(T) \cap \mathbb{R}$, define $\phi_{\alpha}$ as in \eqref{1first} and set $A:=(T-\alpha \mathcal{I})^{-1}$. For $ \breve{f}= \Delta f$ with $f \in \mathcal{SH}_L(\overline{\sigma}_S(T))$ and $f(\alpha)=0$, the operator $ \breve{f}(T)$ defined in \eqref{Ffun} satisfies
$$ \breve{f}(T)= \frac{1}{2 \pi} \int_{\partial(U \cap \mathbb{C}_J)} F_L(s,T) ds_J f(s),$$
where $U$ is any unbounded slice Cauchy domain with $\sigma_S(T) \subset U$ and $ \overline{U} \subset \hbox{dom}(f)$ and $J$ is any imaginary unit $J \in \mathbb{S}$.
A similar integral representation holds for $f \in \mathcal{SH}_R(\overline{\sigma}_S(T))$ with $f(\alpha)=0$.
\end{theorem}
\begin{proof}
Taking into account the relation (\ref{NEWFAP}) the proof follows exactly the same lines of \cite[Thm 8.2.3]{CGK}.
\end{proof}
We observe that, like what happens in the $S$-functional calculus for unbounded operators, even the $F$-functional calculus for unbounded operators does not depend on the parameter $ \alpha$, and the considerations of Remark \ref{RMKDELKER} have to be done in the unbounded case as well.

%The constructions exposed in this subsection can be summarized in the following scheme
%{\small
%\begin{figure}[H]
%	\centering
%	\resizebox{0.80\textwidth}{!}{%
%		\input{Scheme2.tikz}
%	}
%\end{figure}
%}

\section{Harmonic functional calculus for unbounded operators}\label{SECTHARFC}

It is well known that we can factorize the Laplace operator in terms of the Cauchy-Fueter operator $\mathcal{D}$ and we have
$ \Delta=   \mathcal{D} \mathcal{\overline{D}}=  \mathcal{\overline{D}} \mathcal{D}$.
When we apply the operator $ \mathcal{D}$ or its conjugate $ \mathcal{\overline{D}}$ to a slice hyperholomorphic function we obtain different classes of functions. In this section we consider the harmonic fine structure in the sense of Definition \ref{DEFFINE}, see also \cite{CDPS1}.
\begin{definition}[Axially harmonic functions]\label{DEFAXARM}
Let $\Omega \subset \mathbb{R}^4$ be an axially symmetric domain.
A function $f: \Omega \to \mathbb{H}$, with  $f \in \mathcal{C}^2(\Omega)$, is said to be axially harmonic if
$$ \Delta f(q)= \sum_{i=0}^3 \partial_{q_i}^2 f(q)=0$$
and if it has the form \eqref{axial} where the functions $A$ and $B$ satisfy the even-odd conditions \eqref{co1}.
We will denote this function space with the symbol $\mathcal{AH}(\Omega)$.
\end{definition}

When we apply the operator $ \mathcal{D}$  to a slice hyperholomorphic function $f(q)$,  as a direct consequence of the Fueter theorem we get that
$$ \Delta \left(\mathcal{D}f(q)\right)=0.$$
This means that the function $ \mathcal{D}f(q)$ is axially harmonic  in the sense of Definition \ref{DEFAXARM} and
therefore the function spaces of the harmonic fine structure are given by the diagram

\begin{equation}
\label{facto1}
\begin{CD}
	\textcolor{black}{\mathcal{O}(D)}  @>T_{F1}>> \textcolor{black}{\mathcal{SH}(\Omega_D)} @>\mathcal{D}>> \textcolor{black}{\mathcal{AH}(\Omega_D)}  @>\ \    \mathcal{\overline{D}} >>\textcolor{black}{\mathcal{AM}(\Omega_D)},
\end{CD}
\end{equation}
The aim of this section is to define the unbounded functional calculus for the fine structure \eqref{facto1}.
Before we need to recall the definition of the harmonic functional calculus (or $Q$-functional calculus) for bounded
operators based on the $S$-spectrum, see \cite{CDPS1}.

\begin{definition}
\label{Qfun}
Let $T \in \mathcal{BC}(X)$,
assume that the operators $T_{\ell}$, $ \ell=0,1,2,3$ have real spectrum,
  where one of the $T_{\ell}$ is the zero operator
and set $ds_J=ds(-J)$ for $J \in \mathbb{S}$. For every function $ \tilde{f}= \mathcal{D}f$ (resp. $ \tilde{f}= f \mathcal{D}$) with $f \in \mathcal{SH}_L(\sigma_S(T))$ (resp. $f \in \mathcal{SH}_R(\sigma_S(T))$), we set
$$ \tilde{f}(T)=- \frac{1}{\pi} \int_{\partial(U \cap \mathbb{C}_J)} \mathcal{Q}_{c,s}(T)^{-1} ds_J f(s), \quad \left(\hbox{resp.} \quad \tilde{f}(T)=- \frac{1}{\pi} \int_{\partial(U \cap \mathbb{C}_J)}  f(s)ds_J\mathcal{Q}_{c,s}(T)^{-1} \right)$$
where $U$ is an arbitrary bounded slice Cauchy domain with $\sigma_S(T) \subset U$ and $ \overline{U} \subset \hbox{dom}(f)$ and $J \in \mathbb{S}$ is an arbitrary imaginary unit.
\end{definition}
\begin{remark}
The independence of the $Q$-functional calculus from the $ \hbox{ker}(\mathcal{D})$ is obtained by using the monogenic functional calculus of McIntosh. This needs that the operators $T_{\ell}$, $\ell=0,1,2,3$ have real spectrum and one of the $T_{\ell}$ has to be zero (see also the proof of Proposition \ref{invariance} for the case of the unbounded operators).
\end{remark}
The $Q$-functional calculus for left (resp. right) slice hyperholomorphic functions is quaternionic right (resp. left) linear. This means that if we consider $\tilde{f}= \mathcal{D}f$ and $\tilde{g}=\mathcal{D}g$ (resp. $\tilde{f}= f\mathcal{D}$ and $\tilde{g}=g\mathcal{D}$) with $f$, $g \in \mathcal{SH}_L(\sigma_S(T))$ (resp. $f$, $g \in \mathcal{SH}_R(\sigma_S(T))$  and $a \in \mathbb{H}$, for $T \in \mathcal{BC}(X)$ we have
\begin{equation}
\label{line1}
(\tilde{f}a+ \tilde{g})(T)=\tilde{f}(T)a+ \tilde{g}(T), \quad \left(\hbox{resp.} \quad (a\tilde{f}+ \tilde{g})(T)=a \tilde{f}(T)+ \tilde{g}(T) \right).
\end{equation}

In \cite{CDPS1} the $Q$-functional calculus is introduced to get a product rule for the $F$-functional calculus. Moreover, we study a resolvent equation for this functional calculus suitable to generate the Riesz projectors in this setting and to get the following product rule, see \cite{Polyf2}.

\begin{theorem}
Let $T \in \mathcal{BC}(X)$,
assume that the operators $T_{\ell}$, $ \ell=0,1,2,3$ have real spectrum,
  where one of the $T_{\ell}$ is the zero operator. We assume that $f \in \mathcal{N}(\sigma_S(T))$ and $g \in \mathcal{SH}_L(\sigma_S(T))$ or $f \in \mathcal{SH}_R(\sigma_S(T))$ and $g \in \mathcal{N}(\sigma_S(T))$, then
\begin{equation}
\label{prodrule}
\mathcal{D}(fg)(T)=f(T) (\mathcal{D}g)(T)+(\mathcal{D}f)(T)g(T)+ (\mathcal{D}f)(T) \underline{T} (\mathcal{D}g)(T),
\end{equation}
where $ \underline{T}:= T_1e_1+T_2e_2+T_3e_3$.
\end{theorem}

In order to introduce the harmonic functional calculus for unbounded operators
we need to use the homeomorphism $ \phi_{\alpha}$ introduced in Definition \ref{homo}
and its properties in Proposition \ref{simaAT}.

\begin{definition}[Harmonic functional calculus for unbounded operators]\label{HRUNFC}
Let $T \in \mathcal{KC}(X)$,
assume that the operators $T_{\ell}$, $ \ell=0,1,2,3$ have real spectrum,
  where one of the $T_{\ell}$ is the zero operator. Let $ \rho_S(T) \cap \mathbb{R} \neq \emptyset$ and suppose that $ \alpha \in \rho_S(T) \cap \mathbb{R}$.

   For $f \in \mathcal{SH}_L(\overline{\sigma}_S(T))$ we define the functions
$$ \phi(q):= (f \circ \phi_{\alpha}^{-1})(q),$$
$$ \widetilde{\psi}(q):= \mathcal{D}(\phi(q))$$
where $\mathcal{D}$ is the Cauchy-Fueter operator.
The harmonic functional calculus $ \tilde{f}(T)$
is defined as
\begin{equation}
\label{harmo1}
\tilde{f}(T):=(A \bar{A}) \tilde{\psi}(A),
\end{equation}
where $ \tilde{f}=\mathcal{D}f$.
A similar definition holds for $f \in \mathcal{SH}_R(\overline{\sigma}_S(T))$.
\end{definition}

We now use the relation between the two commutative pseudo $S$-resolvent $\mathcal{Q}_{c,p}(A)^{-1}$ and $\mathcal{Q}_{c,s}(T)^{-1}$
in order to provide to the harmonic functional calculus for unbounded operators an integral formulation.

We can prove that the harmonic functional calculus for unbounded operators is well defined because it does not depend on the point $\alpha\in \rho_S(T)$.
Moreover, we also have to prove that the functional calculus is independent from the kernel
of $\mathcal{D}$ taking slice hyperholomorphic functions.

\begin{theorem}
Let $T \in \mathcal{KC}(X)$,
assume that the operators $T_{\ell}$, $ \ell=0,1,2,3$ have real spectrum,
  where one of the $T_{\ell}$ is the zero operator. Let $ \rho_S(T) \cap \mathbb{R} \neq \emptyset$ and suppose that $ \alpha \in \rho_S(T) \cap \mathbb{R}$.
For $ \tilde{f}= \mathcal{D}f$ with $f \in \mathcal{SH}_L(\overline{\sigma}_S(T))$, the operator $ \tilde{f}(T)$ defined in \eqref{harmo1} satisfies
\begin{equation}
\label{harmo2}
\tilde{f}(T)=- \frac{1}{ \pi} \int_{\partial(U \cap \mathbb{C}_J)} \mathcal{Q}_{c,s}(T)^{-1}ds_J f(s),
\end{equation}
where $U$ is any unbounded slice Cauchy domain with $ \sigma_S(T) \subset U$ and $ \bar{U} \subset \hbox{dom}(f)$ and $J \in \mathbb{S}$.
A similar integral representation holds for $f \in \mathcal{SH}_R(\overline{\sigma}_S(T))$.
\end{theorem}
\begin{proof}
We assume that $ \alpha \notin U$. If this is not the case, we can replace the set $U$ with the axially symmetric slice Cauchy domain $U \setminus \overline{B(\alpha)}$ with $ \varepsilon>0$ small enough, without altering the value of the integral \eqref{harmo2} by the Cauchy integral formula.

We set $V= \phi_{\alpha}(U)$. This is a bounded slice Cauchy domain that contain $ \sigma_S(A)$ and $ \overline{V} \subset \hbox{dom}(f \circ \phi_{\alpha}^{-1})$. From the relation among $ \mathcal{Q}_{c,s}(T)^{-1}$ and $ \mathcal{Q}_{c,p}(A)^{-1}$, see Theorem \ref{protQAQT}, and the changing of variables $p= \phi_{\alpha}(s)$ we get
\begin{eqnarray*}
- \frac{1}{ \pi} \int_{\partial(U \cap \mathbb{C}_J)} \mathcal{Q}_{c,s}(T)^{-1} ds_J f(s)&=&- \frac{(A \bar{A})}{ \pi} \int_{\partial(V\cap \mathbb{C}_J)} \mathcal{Q}_{c,p}(A)^{-1}dp_J (f \circ \phi_{\alpha}^{-1})(p)\\
&=& (A \bar{A}) \widetilde{\psi}(A)\\
&=& \tilde{f}(T).
\end{eqnarray*}
\end{proof}

\begin{remark}
Unlike to what happens for the $F$-functional calculus we do not need to require that $f(\alpha)=0$.
\end{remark}
Before we prove that the $Q$-functional calculus is well posed we need the following lemma.
\begin{lemma}
\label{const}
Let assume that $f$, $g \in \mathcal{SH}_L(\overline{\sigma}_S(T))$.
Then we have
$$
\mathcal{D}(f-g)=0\ \ \ \Rightarrow \ \ \ f-g=c,
$$
where $c$ is a constant.

\end{lemma}
\begin{proof}
We observe that since $f$, $g \in \mathcal{SH}_L(\overline{\sigma}_S(T))$ then
$f-g\in \mathcal{SH}_L(\overline{\sigma}_S(T))$ so in particular they are of class $\mathcal{C}^1$.
By hypothesis we have that $f-g$ is a monogenic function that is
\begin{equation}\label{AMF}
f(x_0,r)-g(x_0,r)=\bigl(A(x_0,r)+J\,B(x_0,r)\bigr)P_k(J), \ \ r=|\underline{q}|,
\end{equation}
where $P_k(J)$ are suitable polynomials that depend on the angular part $J\in \mathbb{S}$ and
 $A$ and $B$ are $\mathbb H$-valued continuously differentiable functions in $\mathbb R^2$
 in polar coordinates and they are characterized by the Vekua-type system
\begin{equation}\label{VESo}
\left\{\begin{array}{ll}\partial_{x_0}A-\partial_rB&=
\displaystyle{\frac{2k+m-1}{r}}\,B\\\partial_{x_0}B+\partial_rA&=0.\end{array}\right.
\end{equation}
Monogenic functions of the form (\ref{AMF}) are called axial monogenic of degree $k$, see \cite{DSS}.
We observe that the coefficient $A$ of a slice hyperholomorphic function does not depend on $P_k(J)$, this implies $k=0$ and so that $P_0(J)=1$.
Since $f$, $g \in \mathcal{SH}_L(\overline{\sigma}_S(T))$
\begin{equation}
f(x_0,r)=g(x_0,r)+A(x_0,r)+J\,B(x_0,r)
\end{equation}
also $A(x_0,r)+J\,B(x_0,r)$
must satisfy the Cauchy-Riemann system. Hence we get $B=0$, $\partial_{x_0}A=0$ and $\partial_rA=0$. This implies that $A$ is constant.
\end{proof}
\begin{remark}
In the sequel we will always consider the constant $c$ to be different in each connected components of $ \overline{\sigma}_S(T)$.
\end{remark}
\begin{remark}
\label{bar}
By similar arguments of Lemma \ref{const} it is possible to show, for $f$, $g \in \mathcal{SH}_L(\sigma_{S}(T))$, that
$$
\mathcal{\overline{D}}(f-g)=0\ \ \ \Rightarrow \ \ \ f-g=c,
$$
where $c$ is a constant.
\end{remark}
Let us consider $f$, $f_* \in \mathcal{SH}_L(\overline{\sigma}_S(T))$ such that $ \mathcal{D}(f)= \mathcal{D}(f_*)$. By Lemma \ref{const} we have that $f_*=f+c$. In the next result we show that the harmonic functional calculus is equivalent for $f$ and $f_*$.
\begin{proposition}\label{invariance}
Let $T \in \mathcal{KC}(X)$,
assume that the operators $T_{\ell}$, $ \ell=0,1,2,3$ have real spectrum,
  where one of the $T_{\ell}$ is the zero operator. Let $ \rho_S(T) \cap \mathbb{R} \neq \emptyset$ and suppose that $ \alpha \in \rho_S(T) \cap \mathbb{R}$.
For every function $f \in \mathcal{SH}_L(\overline{\sigma}_S(T))$, the operator $ \tilde{f}(T)$ defined in \eqref{harmo1} does not depend on the choice of $ \alpha \in \rho_S(T) \cap \mathbb{R}$.
Moreover, if we replace $f$ by $f+c$ where $c$ is a quaternionic constant then the $Q$-functional calculus does not depend on $c$.
A similar consideration holds for the integral representation when $f \in \mathcal{SH}_R(\overline{\sigma}_S(T))$.
\end{proposition}
\begin{proof}
The operator defined in \eqref{harmo1} is independent from the parameter $ \alpha \in \rho_S(T) \cap \mathbb{R}$ since the integral in \eqref{harmo2} does not depend on $ \alpha$.
For the second part of the statement let
$
f_*=f+c
$
and we consider
\[
\begin{split}
- \frac{1}{ \pi} \int_{\partial(U \cap \mathbb{C}_J)} \mathcal{Q}_{c,s}(T)^{-1} ds_J f_*(s)
&
=
- \frac{1}{ \pi} \int_{\partial(U \cap \mathbb{C}_J)} \mathcal{Q}_{c,s}(T)^{-1} ds_J (f+c)(s)
\\
&
=
- \frac{1}{ \pi} \int_{\partial(U \cap \mathbb{C}_J)} \mathcal{Q}_{c,s}(T)^{-1} ds_J f(s)
- \frac{1}{ \pi} \int_{\partial(U \cap \mathbb{C}_J)} \mathcal{Q}_{c,s}(T)^{-1} ds_J c
\\
&= (A \bar{A}) \widetilde{\psi}(A)
\\
&= \tilde{f}(T),
\end{split}
\]
where we have used the definition of the harmonic functional calculus for bounded operators, see Definition \ref{Qfun}, and the fact that
\begin{equation}\label{invariance1}
- \frac{1}{ \pi} \int_{\partial(U \cap \mathbb{C}_J)} \mathcal{Q}_{c,s}(T)^{-1} ds_J c=0.
\end{equation}
The last equality is a consequence of two properties of the $F$-resolvent operator and $Q$-resolvent operator. The first property is expressed by the left $F$-resolvent equation  \cite[Thm. 7.3.1]{CGK}
\begin{equation}\label{FREL}
\Q_{c,s}(T)^{-1}=\frac{1}{4}\big(-F_L(s,T)s+TF_L(s,T)\big)
\end{equation}
which implies
$$
- \frac{1}{ \pi} \int_{\partial(U \cap \mathbb{C}_J)} \mathcal{Q}_{c,s}(T)^{-1} ds_J c
=
- \frac{1}{ \pi} \int_{\partial(U \cap \mathbb{C}_J)} \frac{1}{4}\big(-F_L(s,T)s+TF_L(s,T)\big)ds_J c.
$$
The second property is that the following integrals are zero:
$$\int_{\partial(U_i \cap \mathbb{C}_J)} F_L(s,T)s ds_J=0\quad \textrm{and}\quad  \int_{\partial(U_i \cap \mathbb{C}_J)} TF_L(s,T) ds_J=0,$$
where $U_i$'s, for some $q\in \mathbb N$ and $i=1,\dots, q$, are the connected components of $U$. This fact is a consequence of the monogenic functional calculus developed by McIntosh and collaborators (see \cite{JM} and \cite[Remark 7.1.13 and Lemma 7.4.1]{CGK}). We can apply this calculus since we are assuming that the operators $T_{\ell}$, $ \ell=0,1,2,3$ have real spectrum and one of the $T_{\ell}$ is the zero operator.
\end{proof}
\begin{remark}\label{rinvariance}
In the hypothesis of the previous theorem, if  we consider the constant $c$ the same in all the connected components of $U$, we can delete the request that one of the $T_{\ell}$ is the zero operator. Indeed, in this case to prove that
$$
- \frac{1}{ \pi} \int_{\partial(U \cap \mathbb{C}_J)} \mathcal{Q}_{c,s}(T)^{-1} ds_J c=0,
$$
 it is sufficient to apply the Cauchy integral theorem for the left slice hyperholomorphic vector-valued function $\mathcal{Q}_{c,s}(T)^{-1}$.
\end{remark}

We conclude this section by proving a linearity property and a product rule for the $Q$-functional calculus for unbounded operators. These arise naturally as consequences of the respective properties of the $Q$-functional calculus for bounded operators.

\begin{theorem}
Let $T \in \mathcal{KC}(X)$,
assume that the operators $T_{\ell}$, $ \ell=0,1,2,3$ have real spectrum,
  where one of the $T_{\ell}$ is the zero operator. Let $ \rho_S(T) \cap \mathbb{R} \neq \emptyset$ and suppose that $ \alpha \in \rho_S(T) \cap \mathbb{R}$ and $A:= (T- \alpha \mathcal{I})^{-1}$.
  \begin{itemize}
\item (\emph{Linearity}) If $f$, $g \in \mathcal{SH}_L(\overline{\sigma}_S(T))$ and $a \in \mathbb{H}$, then
$$ (\tilde{f}a+\tilde{g})(T)=\tilde{f}(T)a+\tilde{g}(T).$$
Similarly, if $f$, $g \in \mathcal{SH}_R(\overline{\sigma}_S(T))$ and $a \in \mathbb{H}$, then
$$ (a\tilde{f}+\tilde{g})(T)=a\tilde{f}(T)+\tilde{g}(T).$$
\item (\emph{Product rule}) If $f \in \mathcal{N}(\overline{\sigma}_S(T))$ and $g \in \mathcal{SH}_L(\overline{\sigma}_S(T))$ or $f \in \mathcal{SH}_R(\overline{\sigma}_S(T))$ and $g \in \mathcal{N}(\overline{\sigma}_S(T))$, then
$$ \mathcal{D}(fg)(T)= f(T)(\mathcal{D}g)(T)+(\mathcal{D}f)(T) g(T)+ (A \bar{A})^{-1}(\mathcal{D}f)(T) \underline{A} (\mathcal{D}g)(T),$$
where $ \underline{A}:= A_1e_1+A_2e_2+A_3e_3$.
\end{itemize}
\end{theorem}
\begin{proof}
Let $ \alpha \in \rho_S(T) \cap \mathbb{R}$, set $A= (T- \alpha \mathcal{I})^{-1}$, and define $ \phi_{\alpha}$ as in \eqref{1first}. By \eqref{harmo1} and \eqref{line1} we have

\begin{eqnarray*}
(\tilde{f}a+ \tilde{g})(T)&=& (A \bar{A})\mathcal{D} \left((fa+g) \circ \phi_{\alpha}^{-1}\right) (A)\\
&=& (A \bar{A})\mathcal{D}(f \circ \phi_{\alpha}^{-1})(A) a+ (A \bar{A})\mathcal{D}(g \circ \phi_{\alpha}^{-1})(A)\\
&=& \tilde{f}(T)a+ \tilde{g}(T).
\end{eqnarray*}
The statement for right slice hyperholomorphic functions follows from the second equation of \eqref{line1} and similar arguments.

\medskip

To show the product rule we have to use \eqref{harmo1}, \eqref{prodrule} and \eqref{Sfun1}. Since the operator $A\bar A$ is scalar valued, we get

\begin{eqnarray*}
\mathcal{D}(fg)(T)&=& (A \bar{A})\mathcal{D}((fg) \circ \phi_{\alpha})(A)\\
&=& (A \bar{A})\mathcal{D}\left((f \circ \phi_{\alpha}^{-1})(g \circ \phi_{\alpha}^{-1})\right)(A)\\
&=&  (f \circ \phi_{\alpha}^{-1})(A) (A \bar{A}) \mathcal{D} \left( g \circ \phi_{\alpha}^{-1} \right)(A)+ (A \bar{A})\mathcal{D}(f \circ \phi_{\alpha}^{-1}) (A) \left(  g \circ \phi_{\alpha}^{-1} \right) (A)\\
&&+ \mathcal{D}( f \circ \phi_{\alpha}^{-1})(A) \underline{A} (A \bar{A})\mathcal{D}( g \circ \phi_{\alpha}^{-1})(A)\\
&=& f(T)(\mathcal{D}g)(T)+(\mathcal{D}f)(T) g(T)+ (A \bar{A})^{-1}(\mathcal{D}f)(T) \underline{A} (\mathcal{D}g)(T).
\end{eqnarray*}
\end{proof}
\begin{remark}
In order to obtain polyharmonic functional calculus for unbounded operators we need to consider the Fueter-Sce's theorem in the Clifford algebra setting in dimension at least five. In this case we have factorizations of the Fueter extension that take into account polyharmonic functions of any order.
\end{remark}

We  conclude this section with a relation among
the operators $\Q_{c,p}(A)^{-1}$, $\Q_{c,s}(T)^{-1}$ and $S^{-1}_L(s,T)$ that is of independent interest and shows a
 link between the resolvent operators of the harmonic fine structure.
\begin{proposition}\label{REQAQT}
Let $T \in \mathcal{KC}(X)$, let $ \alpha \in \rho_S(T) \cap \mathbb{R}\not=\emptyset$ and define $A:=(T- \alpha \mathcal{I})^{-1}$. For $s \in \rho_S(T)$ and
$p=\phi_{\alpha}(s)=(s- \alpha)^{-1}$, we have
$$
A\Q_{c,p}(A)^{-1}p^3v
=   - S^{-1}_L(s,T)pv +\Q_{c,s}(T)^{-1}v, \ \ \ v\in X.
$$
\end{proposition}
\begin{proof}
By Theorem \ref{protQAQT}
since we assumed that $T\in\mathcal{KC}(X)$ and that there exists a point $\alpha\in \rho_{S}(T)\cap \mathbb{R}\not=\emptyset$, for  $A:=(T-\alpha\id)^{-1}$ and $p =\phi_{\alpha}(s)$, we have
$$
\Q_{c,p}(A)^{-1} = \left( A\overline{A} \right)^{-1}\Q_{c,s}(T)^{-1}p^{-2}.
$$
So we have the chain of equalities
\[
\begin{split}
\Q_{c,p}(A)^{-1} &=  A^{-1} (\overline{A})^{-1}\Q_{c,s}(T)^{-1}p^{-2}
\\
&
=  A^{-1} (\overline{T}-\alpha\id)\Q_{c,s}(T)^{-1}p^{-2}
\\
&
=  A^{-1} (\overline{T}-s\id+s\id-\alpha\id)\Q_{c,s}(T)^{-1}p^{-2}
\end{split}
\]
and also
\[
\begin{split}
\Q_{c,p}(A)^{-1}&=  A^{-1} (-(s\id-\overline{T})+s\id-\alpha\id)\Q_{c,s}(T)^{-1}p^{-2}
\\
&
=  A^{-1} (-(s\id-\overline{T})\Q_{c,s}(T)^{-1}      +(s-\alpha)\Q_{c,s}(T)^{-1})p^{-2}.
\end{split}
\]
Recalling the left $S$-resolvent operator
$$
S^{-1}_L(s,T)=(s \mathcal{I}- \bar{T}) \mathcal{Q}_{c,s}(T)^{-1}
$$
and multiplying on the right for $p^{2}$
we obtain
$$
\Q_{c,p}(A)^{-1}p^2
=   A^{-1}\big( - S^{-1}_L(s,T) +\Q_{c,s}(T)^{-1}(s-\alpha)\big).
$$
Taking into account the relation $p =\phi_{\alpha}(s)$ we get
$$
\Q_{c,p}(A)^{-1}p^2
= A^{-1}\big( - S^{-1}_L(s,T) +\Q_{c,s}(T)^{-1}p^{-1}\big).
$$
Multiplying one more time on the right by $p$  we finally get
$$
A\Q_{c,p}(A)^{-1}p^3v
=   - S^{-1}_L(s,T)pv +\Q_{c,s}(T)^{-1}v.
$$
The manipulation on $A$ are justified by the fact that it is a bounded operator and the above relation holds for every $v\in X$.
\end{proof}

\section{Polyanalytic functional calculus for unbounded operators}

As we have discussed in Section \ref{SECTHARFC},
although $ \Delta= \mathcal{D} \mathcal{\overline{D}}=\mathcal{\overline{D}} \mathcal{D} $,
if we apply the conjugate Cauchy-Fueter operator $ \mathcal{\overline{D}}$ to a slice hyperholomorphic function $f(q)$
we get a different set of functions with respect to the harmonic case.
In fact, as a direct consequence of the Fueter's theorem we have
$$ \mathcal{D}^2 \left(\mathcal{\overline{D}}f(q)\right)= \Delta \mathcal{D} f(q)=0.$$
So we obtain a second
sequence of function spaces for the quaternionic fine structure
\begin{equation}
\label{facto2}
\begin{CD}
	\textcolor{black}{\mathcal{O}(D)}  @>T_{F1}>> \textcolor{black}{\mathcal{SH}(\Omega_D)} @>\mathcal{\overline{D}}>> \textcolor{black}{\mathcal{AP}_2(\Omega_D)}  @>\ \  \mathcal{D} >>\textcolor{black}{\mathcal{AM}(\Omega_D)},
\end{CD}
\end{equation}
where $ \mathcal{AP}_2(\Omega_D)$ is the set of axially polyanalytic functions of order $2$.

\begin{definition}[Axially polyanalytic functions of order $2$]
Let $\Omega \subset \mathbb{R}^4$ be an axially symmetric slice domain. A function $f: \Omega \to \mathbb{H}$ is said to be (left) axially polyanalytic of order $2$ if $f \in \mathcal{C}^2(\Omega)$ and
$$ \mathcal{D}^2 f(q)= \left (\partial_{q_0}+ \sum_{i=1}^3 e_i \partial_{q_i}\right)^2 f(q)=0,$$
moreover $f$ has the form \eqref{axial} and the functions $A$ and $B$ satisfy the even-odd conditions \eqref{co1}.
\end{definition}
We now focus our attention on the polyanalytic functional calculus for unbounded operators of this fine structure.
In order to do this we need the basic notions for this functional calculus for bounded operators, see \cite{Polyf1} for more details.
\begin{definition}[$P_2$-resolvents]\label{DEFP2res}
Let $T=T_0+ \sum_{i=1}^3 e_i T_i \in \mathcal{BC}(X)$, $s \in \mathbb{H}$, we define the left (resp. right) $P_2$-resolvent operator as
$$ P_2^L(s,T)=-F_L(s,T)s+T_0 F_{L}(s,T), \quad \left(\hbox{resp.} \quad P_2^R(s,T)=-sF_R(s,T)+T_0 F_{R}(s,T) \right).$$
\end{definition}

We recall that the polyanalytic functional calculus for bounded operators is also called $P_2$-functional calculus.

\begin{definition}[Polyanalytic functional calculus for bounded operators]
Let $T \in \mathcal{BC}(X)$,
assume that the operators $T_{\ell}$, $ \ell=0,1,2,3$ have real spectrum,
  where one of the $T_{\ell}$ is the zero operator,
  and set $ds_J=ds(-J)$ for $J \in \mathbb{S}$. For every function $ \breve{f}^\circ= \mathcal{\overline{D}} f$ (resp. $\breve{f}^{\circ}= f \mathcal{\overline{D}}$ ) with $f \in \mathcal{SH}_L(\sigma_S(T))$ (resp. $f \in \mathcal{SH}_R(\sigma_S(T))$),
  we define
$$ \breve{f}^{\circ}(T):= \frac{1}{2 \pi} \int_{\partial(U \cap \mathbb{C}_J)} P_2^L(s,T) ds_J f(s), \quad \left(\hbox{resp.} \quad \breve{f}^{\circ}(T):= \frac{1}{2 \pi} \int_{\partial(U \cap \mathbb{C}_J)} f(s) ds_J P_2^R(s,T)\right),$$
where $F_L(s,T)$ (resp. $F_R(s,T)$) is the left (resp. right) resolvent operator defined in \eqref{star1}.
\end{definition}
\begin{remark}
Similarly to the $Q$-functional calculus and $F$-functional calculus the independence of the polyanalytic functional calculus from the $ \hbox{ker}(\mathcal{\overline{D}})$ follows by the monogenic functional calculus of McIntosh. This requires that the operators $T_{\ell}$, $\ell=0,1,2,3$ have real spectrum and one of the $T_{\ell}$ has to be zero.
\end{remark}
\begin{remark}
As it happens for other functional calculi based on the $S$-spectrum,
also the polyanalytic functional calculus enjoys a linearity property.
If $ \breve{f}^{\circ}= \mathcal{\overline{D}}f$ and $ \breve{g}^{\circ}= \mathcal{\overline{D}} g$ (resp. $\breve{f}^{\circ}= f\mathcal{\overline{D}}$ and $ \breve{g}^{\circ}=g \mathcal{\overline{D}} $) with $f$, $g \in \mathcal{SH}_L(\sigma_S(T))$ (resp.  $f$, $g \in \mathcal{SH}_R(\sigma_S(T))$ ) and $a \in \mathbb{H}$ then
\begin{equation}
\label{polylin}
(\breve{f}^{\circ}a+ \breve{g}^{\circ})(T)=\breve{f}^{\circ}(T)a+  \breve{g}^{\circ}(T) \quad \left(\hbox{resp.} \quad (a\breve{f}^{\circ}+ \breve{g}^{\circ})(T)=a\breve{f}^{\circ} (T)+  \breve{g}^{\circ}(T) \right).
\end{equation}
\end{remark}
Considering the transformation from unbounded operators $T \in \mathcal{KC}(X)$ to bounded operators $A$
defined as
$A:=(T- \alpha \mathcal{I})^{-1}$,
we now show a relation between the resolvents $P_2^L(p,A)$ and $P_2^{L}(s,T)$
 that will lead us to the suitable definition of the $P_2$-functional calculus.

\begin{theorem}\label{t1}
	Let $T\in\mathcal{KC}(X)$ for $\alpha \in\rho_S(T)\cap \mathbb R\neq \emptyset$ we define $A=(T-\alpha\mathcal I)^{-1}$. Assume that $s\in\rho_S(T)$ and $p=(s-\alpha)^{-1}$, then we have
	\begin{equation}\label{f1}
		P^L_2(p,A)p^4=P^L_2(s, T)-4(s\mathcal I-\bar T)\mathcal Q_{c,s}(T)^{-1}p+4(T_0-s\mathcal I)\mathcal Q_{c,s}(T)^{-1}p+4p^2\mathcal I
	\end{equation}
where (\ref{f1}) holds on $X$.
\end{theorem}

\begin{proof}
By Theorem \ref{protQAQT} we know that
$$
\Q_{c,p}(A)^{-1} = \left( A\overline{A} \right)^{-1}\Q_{c,s}(T)^{-1}p^{-2},\ \ {\rm for} \ \ s\in\rho_S(T)
$$
and from the definition  of the left $P_2$-resolvent and of the $F$-resolvent operators we get
	\[
		\begin{split}
			P^L_2(p,A)&=\frac{(A+\bar A)}{2}F_L(p,A)-F_L(p,A)p\\
			&=-4\frac{(A+\bar A)}{2} \Big((p\mathcal I-\bar A)\mathcal Q_{c,p}(A)^{-2}\Big)+
4\Big((p\mathcal I-\bar A)\mathcal Q_{c,p}(A)^{-2}\Big)p\\
			&=-4(p\mathcal I-\bar A)\frac{(A+\bar A)}{2}\Big((A\bar A)^{-2}\mathcal Q_{c,s}(T)^{-2}p^{-4}\Big)
+4(p\mathcal I-\bar A) \Big((A\bar A)^{-2}\mathcal Q_{c,s}(T)^{-2}p^{-4}\Big)p\\
			&=4(A\bar A)^{-1}(p\mathcal I-\bar A)\left(-\frac{(A\bar A)^{-1}(A+\bar A)}{2}+p(A\bar A)^{-1}\right) \mathcal Q_{c,s}(T)^{-2}p^{-4}\\
			&=4(A\bar A)^{-1}(p\mathcal I-\bar A)\left(-\frac{A^{-1}+\bar A^{-1}}{2}+p(A\bar A)^{-1}\right) \mathcal Q_{c,s}(T)^{-2}p^{-4}.
		\end{split}
	\]
We observe that
$$
Q_{c,s}(T)^{-2}:\  X \to \hbox{dom}((T\overline{T})^2), \ \ {\rm for} \ \  s\in\rho_S(T).
$$
Moreover, we recall that  $\hbox{dom}(T)=\hbox{dom}(\overline{T})$ so $\hbox{dom}((T\overline{T})^2)=\hbox{dom}(T^4)$, and
$$
(A\bar A)^{-1}:\ \hbox{dom}(T\overline{T})\to X.
$$
Therefore, for $s\in\rho_S(T)$ we have
$$
(A\bar A)^{-1} \mathcal Q_{c,s}(T)^{-2}: \ X\to \hbox{dom}(T\overline{T}),
$$
and
$$
\left(-\frac{A^{-1}+\bar A^{-1}}{2}\right) \mathcal Q_{c,s}(T)^{-2}:\ X\to \hbox{dom}(T^3).
$$
Furthermore, we have that
$$
(A\bar A)^{-1}(p\mathcal I-\bar A):\ \hbox{dom}(T\overline{T})\to X.
$$
Hence for $s\in\rho_S(T)$ we get
\begin{equation}
\label{CIOEDF}
P^L_2(p,A)v=4(A\bar A)^{-1}(p\mathcal I-\bar A)\left(-\frac{A^{-1}+\bar A^{-1}}{2}+p(A\bar A)^{-1}\right) \mathcal Q_{c,s}(T)^{-2}p^{-4}v, \ \
\hbox{for all}\ \  v\in X.
\end{equation}

On $\hbox{dom}(T\overline{T})$ it is $(A\bar A)^{-1}(p\mathcal I-\bar A)=(p\mathcal I-\bar A)(A\bar A)^{-1}$, so
from the assumption $p=\Phi_\alpha(s)$, we have $\alpha=s-p^{-1}$ and
\begin{eqnarray}
\nonumber
(A\bar A)^{-1}(p\mathcal I-\bar A)& =&(p\mathcal I-(\bar T-\alpha I)^{-1})(\bar T-\alpha\mathcal I)(T-\alpha\mathcal I)\\
\nonumber
&=&p(\alpha^2\mathcal I-\alpha(T+\bar T)+T\bar T)+\alpha\mathcal I-T\\
\nonumber
&=& p(s^2\mathcal I+p^{-2}\mathcal I-2sp^{-1}\mathcal I-s(T+\bar T)+p^{-1}(T+\bar T)+T\bar T)\\
\nonumber
&& +s\mathcal I-p^{-1}\mathcal I-T\\
\label{eq1}
&=& p\mathcal Q_{c,s}(T)-(s\mathcal I-\bar T),
\end{eqnarray}
so for all $ v\in\hbox{dom}(T \bar{T})$ we have
$$
(A\bar A)^{-1}(p\mathcal I-\bar A)v=p\mathcal Q_{c,s}(T)v-(s\mathcal I-\bar T)v,
$$
$$
P^L_2(p,A)v=4\left(p\mathcal Q_{c,s}(T)-(s\mathcal I-\bar T)\right)\left(-\frac{A^{-1}+\bar A^{-1}}{2}+p(A\bar A)^{-1}\right) \mathcal Q_{c,s}(T)^{-2}p^{-4}v.
$$
With some computations we have
\begin{equation}
\label{star2}
\left(-\frac{A^{-1}+\bar A^{-1}}{2}+p(A\bar A)^{-1}\right)=-s\mathcal I+T_0+p\mathcal Q_{c,s}(T),
\end{equation}
which holds on $\hbox{dom}(T\overline{T})$.
By formula \eqref{eq1}, \eqref{CIOEDF} and \eqref{star2}, we have
$$
P^L_2(p,T)=4[p\mathcal Q_{c,s}(T)-(s\mathcal I-\bar T)](-s\mathcal I+p\mathcal Q_{c,s}(T)+T_0)\mathcal Q_{c,s}(T)^{-2}p^{-4}.
$$
Now, since $(-s\mathcal I+p\mathcal Q_{c,s}(T)+T_0)$ and $\mathcal Q_{c,s}(T)^{-2}$ commute, we have
	\[
	\begin{split}
		 P^L_2(p,T) &=4(p\mathcal Q_{c,s}(T)-s\mathcal I+\bar T)\mathcal Q_{c,s}(T)^{-2}(-s\mathcal I+T_0+p\mathcal Q_{c,s}(T))p^{-4}\\
		&=(4p\mathcal Q_{c,s}(T)^{-1}+F_L(s,T))(-s\mathcal I+T_0+p\mathcal Q_{c,s}(T))p^{-4}\\
		&=(-F_L(s,T)s+T_0F_L(s,T))p^{-4}+F_L(s,T)p\mathcal Q_{c,s}(T)p^{-4}\\
		& \quad+4(-s\mathcal I+T_0)\mathcal Q_{c,s}(T)^{-1}p^{-3}+4p^{-2}\\
		&=P^L_2(s,T)p^{-4}-4(s\mathcal I-\bar T)\mathcal Q_{c,s}(T)^{-1}p^{-3}+4(-s\mathcal I+T_0)\mathcal Q_{c,s}(T)^{-1}p^{-3}+4p^{-2}.
	\end{split}
	\]
	By multiplying on the right hand side by $p^4$ we get the statement.
\end{proof}
\begin{remark}
	It is possible to prove Theorem \ref{t1} starting from other two different ways to write the resolvent operator $P^L_2(p,A)$. The first one is the following
	
	$$ P^L_2(p,A)=\frac{A+\bar A}{2}F_L(p,A)-F_L(p,A)p$$
	
	and, the other one is this
	
	$$ P^L_2(p,A)=-4S^{-1}_L(p,A)\mathcal Q_{c,p}(A)^{-1}(A_0-p\mathcal I). $$
	
	To obtain the desired formula, in the first case it is sufficient to use the relation $F_L(p,A)=(A\bar A)^{-1}(-4p\mathcal Q_{c,s}(T)^{-1}-F_L(s,T))p^{-4}$, while, in the second case, it is sufficient to use the relations: $S^{-1}_L(p,A)=p^{-1}\mathcal I-S^{-1}_L(s,T)p^{-2}$ and $\mathcal Q_{c,p}(A)^{-1}=(A\bar A)^{-1}\mathcal Q_{c,s}(T)^{-1}p^{-2}$.
\end{remark}

Now, we are ready to give a definition for a polyanalytic functional calculus of order $2$ for unbounded operators.

\begin{definition}[Polyanalytic Functional Calculus of order 2 for unbounded operators]\label{d1}

Let $T \in \mathcal{KC}(X)$,
assume that the operators $T_{\ell}$, $ \ell=0,1,2,3$ have real spectrum,
  where one of the $T_{\ell}$ is the zero operator. Let $ \rho_S(T) \cap \mathbb{R} \neq \emptyset$, suppose that $ \alpha \in \rho_S(T) \cap \mathbb{R}$ and set $A:= (T- \alpha \mathcal{I})^{-1}$.
Assume that $\phi_\alpha$ is as in \eqref{1first}. For $f\in \mathcal{SH}_L(\bar\sigma_S(T))$ with $f(\alpha)=0$ and $\partial_{\alpha} f(\alpha)=0$, we consider the functions
	$$ \phi(q):=(f\circ \phi_\alpha^{-1})(q) $$
	
	$$ \breve{\psi}^0(q):= \mathcal{\overline{D}}(q^2\phi(q))$$
	
	and we define the operator $\breve{f}^0(T)$, for $\breve{f}^0= \mathcal{\overline{D}}f$, as
\begin{equation}
\label{polyFun}
\breve f^0(T):=\breve \psi^0(A),
\end{equation}
	where $\breve\psi^0(A)$ is defined via the bounded polyanalytic functional calculus.
A similar definition holds for $f\in \mathcal{SH}_R(\bar\sigma_S(T))$ with $f(\alpha)=0$ and $\partial_{\alpha} f(\alpha)=0$.
\end{definition}

Even for the polyanalytic functional calculus for unbounded operators we can have an integral representation.
\begin{theorem}
	Let $T \in \mathcal{KC}(X)$,
assume that the operators $T_{\ell}$, $ \ell=0,1,2,3$ have real spectrum,
  where one of the $T_{\ell}$ is the zero operator. Let $ \rho_S(T) \cap \mathbb{R} \neq \emptyset$, suppose that $ \alpha \in \rho_S(T) \cap \mathbb{R}$ and set $A:= (T- \alpha \mathcal{I})^{-1}$.
For $\breve f^0=\mathcal{\overline{D}}f$ with $f\in \mathcal{SH}_L(\bar\sigma(T))$ with $f(\alpha)=0$ and $\partial_{\alpha}f(\alpha)=0$, the operator defined in \eqref{polyFun} satisfies
\begin{equation}
\label{inteversion}
\breve{f}^0(T)= \frac{1}{2 \pi}\int_{\partial(U\cap\mathbb C_J)} P^L_2(s,T) \, ds_J\, f(s)
\end{equation}
	where $U$ is any unbounded slice Cauchy domain with $\bar \sigma_S(T)\subset U$ and $\bar U \subset \hbox{dom}(f)$ and $J$ is any imaginary unit in $\mathbb S$.
A similar integral representation holds for $f\in \mathcal{SH}_R(\bar\sigma_S(T))$ with $f(\alpha)=0$ and $\partial_{\alpha} f(\alpha)=0$.
\end{theorem}
\begin{proof}
	We assume that $\alpha\notin \bar U$. If this is not the case, we can replace $U$ by the axially symmetric slice Cauchy domain $U\setminus \overline B_{\epsilon}(\alpha)$ with an $\epsilon>0$ small enough, without altering the value of the integral by the Cauchy integral formula.
	
	The set $V=\phi_\alpha(U)$ is a bounded slice Cauchy domain with $\bar\sigma_S(T)\subset V$ and $\bar V\subset \operatorname{dom}(f\circ \phi_\alpha^{-1})$. Using the relation between $P^L_2(p,A)$ and $P^L_2(s,T)$ (see formula \eqref{f1}) we have
	\begin{equation}\label{f2}
		\begin{split}
			& \int_{\partial(U\cap\mathbb C_J)} (P_2^L(s,T)-4(s\mathcal I-\bar T)\mathcal Q_{c,s}(T)^{-1}p+4(T_0-s\mathcal I)\mathcal Q_{c,s}(T)^{-1}p+4p^2\mathcal I)\, ds_J\, f(s)\\
			&= \int_{\partial (V\cap\mathbb C_J)}P^L_2(p,A)\, dp_J\, p^2\phi(p),
		\end{split}
	\end{equation}
where we recall that $s$ and $p$ are related by $p=\phi_{\alpha}(s)=(s- \alpha)^{-1}$.
	Now we focus on the left hand side of the previous equation, we have
	\begin{equation}\label{f3}
		\begin{split}
			&\int_{\partial(U\cap\mathbb C_J)} (P_2^L(s,T)-4(s\mathcal I-\bar T)\mathcal Q_{c,s}(T)^{-1}p+4(T_0-s\mathcal I)\mathcal Q_{c,s}(T)^{-1}p+4p^2\mathcal I)\, ds_J\, f(s)\\
			&= \int_{\partial(U\cap\mathbb C_J)} P^L_2(s,T)\, ds_J\, f(s) - 4\int_{\partial (U\cap\mathbb C_J)} (s\mathcal I-\bar T)\mathcal Q_{c,s}(T)^{-1} \, ds_J\, pf(s)\\
			& +4\int_{\partial(U\cap\mathbb C_J)} p\, ds_J\, (T_0-s\mathcal I)\mathcal Q_{c,s}(T)^{-1} f(s)+4\int_{\partial(U\cap\mathbb C_J)} p^2\, ds_J\, f(s).
		\end{split}
	\end{equation}
	We consider the last three terms in the left hand side of the previous equation. By using the fact that $p=(s-\alpha)^{-1}$ we get
	\[
	\begin{split}
		& \int_{\partial (U\cap\mathbb C_J)}(s\mathcal I-\bar T) \mathcal Q_{c,s}(T)^{-1}\, ds_J\, pf(s)\\
		& =\int_{\partial (U\cap\mathbb C_J)} p\, ds_J\, s\mathcal Q_{c,s}(T)^{-1} f(s)-\bar T \int_{\partial (U\cap\mathbb C_J)} p\, ds_J\, \mathcal Q_{c,s}(T)^{-1}f(s)\\
		&= \int_{\partial (U\cap\mathbb C_J)} (s-\alpha)^{-1}\, ds_J\, s\mathcal Q_{c,s}(T)^{-1} f(s)-\bar T\int_{\partial (U\cap\mathbb C_J)} (s-\alpha)^{-1}\, ds_J\, \mathcal Q_{c,s}(T)^{-1} f(s),
	\end{split}
	\]
since the operator $\bar{T}$ is closed so it commute with the integral.
	Observe that the functions $s\mapsto s\mathcal Q_{c,s}(T)^{-1}$ and $s\mapsto \mathcal Q_{c,s}(T)^{-1}$ are operator-valued left slice hyperholomorphic, by the vector valued Cauchy formula \cite[see Theorem 2.3.19]{CGK} we get
	\[
	\begin{split}
		&\int_{\partial (U\cap\mathbb C_J)}(s\mathcal I-\bar T)\mathcal Q_{c,s}(T)^{-1} \, ds_J\, pf(s)\\
		&=\int_{\partial (U\cap\mathbb C_J)} S^{-1}_L(s,\alpha) \, ds_J\, s\mathcal Q_{c,s}(T)^{-1}f(s)-\bar T\int_{\partial (U\cap\mathbb C_J)} S^{-1}_L(s,\alpha)\, ds_J\, \mathcal Q_{c,s}(T)^{-1} f(s)\\
		&=(2\pi)\alpha\mathcal Q_{c,\alpha}(T)^{-1}f(\alpha)-(2\pi)\bar{T}\mathcal Q_{c,\alpha}(T)^{-1}f(\alpha).
	\end{split}
	\]
	By following similar arguments we get
	\[
	\begin{split}
		& \int_{\partial(U\cap\mathbb C_J)}p\, ds_J\, (T_0-s\mathcal I)\mathcal Q_{c,s}(T)^{-1}f(s)\\
		&= T_0 \int_{\partial(U\cap\mathbb C_J)} S^{-1}_L(s,\alpha)\, ds_J\, \mathcal Q_{c,s}(T)^{-1} f(s) - \int_{\partial(U\cap\mathbb C_J)} S^{-1}_L(s,\alpha) \, ds_J\, s\mathcal Q_{c,s}(T)^{-1}\\
		&= (2\pi)T_0\mathcal Q_{c,\alpha}(T)^{-1} f(\alpha)-(2\pi)\alpha\mathcal Q_{c,\alpha}(T)^{-1} f(\alpha).
	\end{split}
	\]
	To work with the last term of \eqref{f3} we need slightly different manipulations
	\[
	\begin{split}
		\int_{\partial(U\cap\mathbb C_J)} p^2\, ds_J\, f(s)&=\int_{\partial(U\cap\mathbb C_J)} (s-\alpha)^{-2}\, ds_J\, f(s)=\partial_{\alpha}\int_{\partial(U\cap\mathbb C_J)} (s-\alpha)^{-1} \, ds_J\, f(s)\\
		&= \partial_{\alpha} \int_{\partial(U\cap\mathbb C_J)} S^{-1}_L(s,\alpha) \, ds_J\, f(s)=(2\pi)\partial_{\alpha} f(\alpha).
	\end{split}
	\]
	The identity \eqref{f2} turns into
	\[
	\begin{split}
		& -(8\pi)\alpha\mathcal Q_{c,\alpha} (T)^{-1} f(\alpha)+(8\pi)\bar T\mathcal Q_{c,\alpha}(T)^{-1} f(\alpha) +(8\pi) T_0\mathcal Q_{c,\alpha}(T)^{-1} f(\alpha)\\
		& -(8\pi)\alpha\mathcal Q_{c,\alpha}(T)^{-1} f(\alpha)+8\pi\partial_{\alpha} f(\alpha) +\int_{\partial(U\cap\mathbb C_J)}P^L_2(s,T)\, ds_J\, f(s)= \int_{\partial(U\cap\mathbb C_J)} P^L_2(p,A)\, dp_J\, p^2\phi(p).
	\end{split}
	\]
	Since by assumption we have $f(\alpha)=0$ and $\partial_{\alpha} f(\alpha)=0$, we get
	\[
	\frac{1}{2\pi} \int_{\partial(U\cap\mathbb C_J)} P^L_2(s,T) \, ds_J\, f(s) = \frac 1{2\pi}\int_{\partial(U\cap\mathbb C_J)} P^L_2(p,A)\, dp_J p^2 \phi(p)\\
	=\breve{\psi_0} (A)=\breve f_0(T),
	\]
and so we get the statement.
\end{proof}
Let us consider $f$, $f_* \in \mathcal{SH}_L(\overline{\sigma}_S(T))$ such that $\mathcal{\overline{D}}(f)= \mathcal{\overline{D}}(f_{*})$. Then by Remark \ref{bar} we have that there exists a constant $c$ such that $ f_*=f+c$. In the next result we show that the polyanalytic functional calculus for $f$ and $f_{*}$ is equivalent.

\begin{proposition}
Let $T \in \mathcal{KC}(X)$,
assume that the operators $T_{\ell}$, $ \ell=0,1,2,3$ have real spectrum,
  where one of the $T_{\ell}$ is the zero operator. Let $ \rho_S(T) \cap \mathbb{R} \neq \emptyset$, suppose that $ \alpha \in \rho_S(T) \cap \mathbb{R}$.
For every function $f \in \mathcal{SH}_L(\overline{\sigma}_S(T))$ with $f(\alpha)=0$ and $\partial_{\alpha}f(\alpha)=0$, the operator $ \breve{f}^0(T)$ defined in \eqref{polyFun} does not depend on the choice of $ \alpha \in \rho_S(T) \cap \mathbb{R}$.
Moreover, if we replace $f$ by $f_*=f+c$, where $c$ is a quaternionic constant, such that $f_*(\beta)=0$ and $\partial_{\beta} f_*(\beta)=0$, for $\beta \in \rho_S(T) \cap \mathbb{R}$. Then the $P_2$-functional calculus does not depend on $c$.
Similar considerations hold for $f \in \mathcal{SH}_R(\overline{\sigma}_S(T))$ with $f(\alpha)=0$ and $\partial_{\alpha}f(\alpha)=0$.
\end{proposition}
\begin{proof}
The operator defined in \eqref{polyFun} is independent from the parameter $ \alpha \in \rho_S(T) \cap \mathbb{R}$ since the integral in \eqref{inteversion} does not depend on $ \alpha$.

Let us consider $f_*:=f+c$, where $c$ is a generic quaternion. We suppose that for some $ \beta \in \rho_S(T) \cap \mathbb{R}$ we have $f_{*}(\beta)=0$ and $\partial_{\beta}f_{*}(\beta)=0$. We observe that $ \breve{f}_*^0= \breve{f}^0$. From the integral representation \eqref{inteversion} we have
\begin{eqnarray}
\breve{f_{*}^0}(T)&=& \frac{1}{2 \pi} \int_{\partial(U \cap \mathbb{C}_J)} P_2^L(s,T) ds_J f_{*}(s)\\
&=& \frac{1}{2 \pi} \int_{\partial(U \cap \mathbb{C}_J)} P_2^L(s,T) ds_J f(s)+\frac{1}{2 \pi} \int_{\partial(U \cap \mathbb{C}_J)} P_2^L(s,T) ds_J c.
\end{eqnarray}
Since $P_2^L(s,T)=-F_L(s,T)s+T_0 F_{L}(s,T)$, we can prove that:
$$ \frac{1}{2 \pi}\int_{\partial(U \cap \mathbb{C}_J)} P_2^L(s,T) ds_J c=0,$$
using the same argument that we considered at the end of Proposition \ref{invariance} to prove equation \eqref{invariance1}. Then we get
$$ \breve{f}_*^0(T)= \breve{f}^0(T).$$
This means that we can define $ \breve{f}^0(T)$ by means of $ \beta$ instead of $ \alpha$.
\end{proof}

\begin{remark}
As we observed in Remark \ref{rinvariance}, in the hypothesis of the previous theorem, if the constant $c$ is the same in all the connected components of $U$, we can delete the request that one of the $T_{\ell}$ is the zero operator. Indeed, in this case to prove that
$$
- \frac{1}{ \pi} \int_{\partial(U \cap \mathbb{C}_J)} P^L_2(s,T) ds_J c=0,
$$
 it is sufficient to apply the Cauchy integral theorem for the left slice hyperholomorphic vector-valued function $P^L_2(s,T)$.
\end{remark}

Now, we show two important properties for the $P_2$-functional calculus for unbounded operators.

\begin{theorem}
Let $T \in \mathcal{KC}(X)$ with $\rho_S(T) \cap \mathbb{R} \neq \emptyset$. We assume that the operators $T_{\ell}$, $ \ell=0,1,2,3$ have real spectrum,
where one of the $T_{\ell}$ is the zero operator.
\begin{itemize}
\item(Linearity) If $f$, $g \in \mathcal{SH}_L(\overline{\sigma}_S(T))$ and $a \in \mathbb{H}$, then
$$ (\breve{f}^0a+\breve{g}^0)(T)= \breve{f}^0(T)a+ \breve{g}^0(T).$$
Similarly, if $f$, $g \in \mathcal{SH}_R(\overline{\sigma}_S(T))$ and $a \in \mathbb{H}$, then
$$ (a\breve{f}^0+\breve{g}^0)(T)= a\breve{f}^0(T)+ \breve{g}^0(T).$$
\item (Product rule)  If $f \in \mathcal{N}(\overline{\sigma}_S(T))$ and $g \in \mathcal{SH}_L(\overline{\sigma}_S(T))$ or $f \in \mathcal{SH}_R(\overline{\sigma}_S(T))$ and $g \in \mathcal{N}(\overline{\sigma}_S(T))$, then
$$ \mathcal{\overline{D}}(fg)(T)=f(T)(\mathcal{\overline{D}}g)(T)+ (\mathcal{\overline{D}}f)(T) g(T)- \mathcal{D}(f)(T) \underline{T} \mathcal{D}(g)(T)$$\end{itemize}
\end{theorem}
\begin{proof}
Let us consider $ \alpha \in \rho_S(T) \cap \mathbb{R}$ and $A= (T- \alpha \mathcal{I})^{-1}$. We recall that $ \phi_{\alpha}$ is defined in \eqref{1first}. By \eqref{polyFun} and \eqref{polylin} we have
\begin{eqnarray*}
	(a\breve{f}^0+\breve{g}^0)(T)&=& \mathcal{\overline{D}} \left(q^2(fa+g) \circ \phi_{\alpha}^{-1}\right) (A)\\
	&=& \mathcal{\overline{D}} \left( q^2(f \circ \phi_{\alpha}^{-1})a+q^2(g \circ \phi_{\alpha}^{-1})\right)(A)\\
	&=& \mathcal{\overline{D}}(q^2(f \circ \phi_{\alpha}^{-1})) (A) a+ \mathcal{\overline{D}}(q^2(g \circ \phi_{\alpha}^{-1}))(A)\\
	&=& \breve{f}^0(T)a+ \breve{g}^0(T).
\end{eqnarray*}
By similar arguments it is possible to show the statement for right slice hyperholomorphic functions.

\medskip

The product rule follows by analogous computations done to obtain the product rule of the polyanalytic functional calculus of order 2 for bounded operators, see \cite[Thm. 4.6]{Polyf2}. These are based on the $S$-resolvent equation which holds also for unbounded operators, see Theorem \ref{proper}.
\end{proof}

The definition and the properties of the polyanalytic functional calculus can be extended to the case of $n$-tuples of unbounded operators and for polyanalytic function of greater order. In order to do this we need to work with the Fueter's theorem in the Clifford algebra setting in dimension at least five, because we will have more involved factorizations than the ones considered in \eqref{facto1} and \eqref{facto2}.

\section{Concluding remarks}

In \cite{CDPS1, Polyf1}, based on the factorizations \eqref{facto1} and \eqref{facto2} the authors have studied harmonic and polyanalytic functional calculi based on the $S$-spectrum. These are based on the integral transforms given in \eqref{inteharmo} and \eqref{intepoly}.
The quaternionic fine structures in the cases of bounded and unbounded operators  induced by the factorization of the Laplace operator
in terms of the Cauchy-Fueter operator and of its conjugate can be globally summarized in the following diagram
{\small
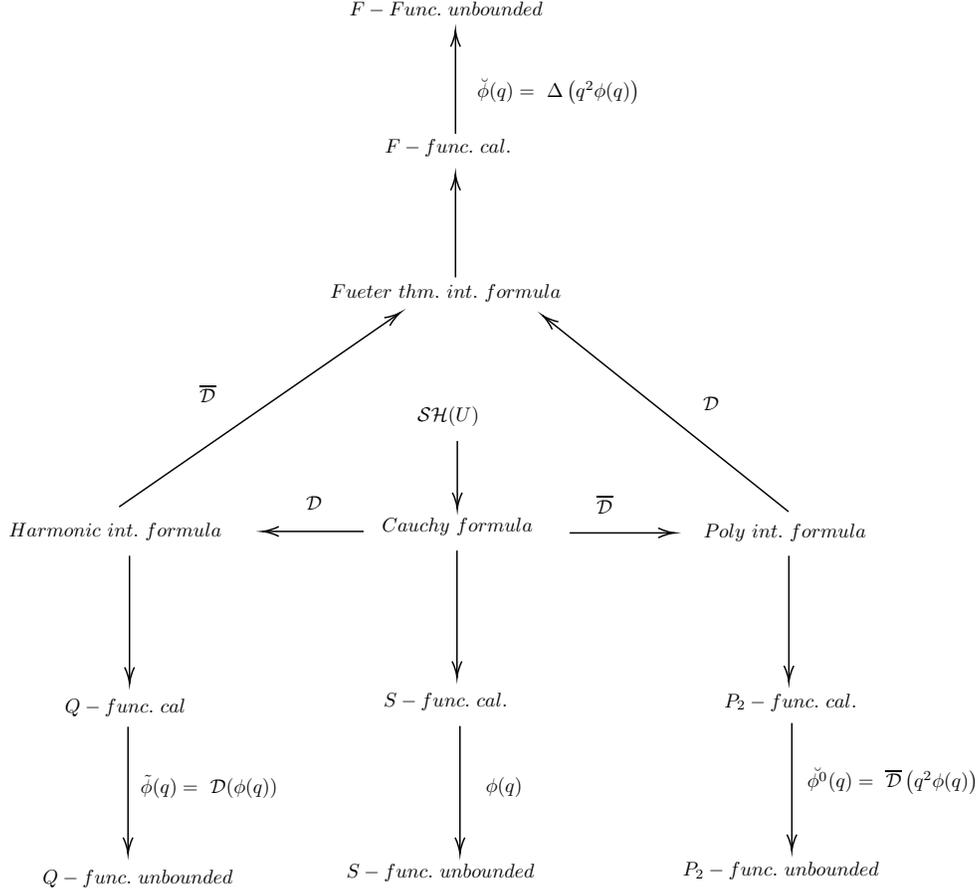
\begin{figure}[H]
	\centering
	\resizebox{0.80\textwidth}{!}{%
		\tikzset{every picture/.style={line width=0.75pt}} %set default line width to 0.75pt        

\begin{tikzpicture}[x=0.75pt,y=0.75pt,yscale=-1,xscale=1]
	%uncomment if require: \path (0,682); %set diagram left start at 0, and has height of 682
	
	%Straight Lines [id:da3759542931898665] 
	\draw    (417.22,335.98) -- (417.22,379.83) ;
	\draw [shift={(417.22,381.83)}, rotate = 270] [color={rgb, 255:red, 0; green, 0; blue, 0 }  ][line width=0.75]    (10.93,-3.29) .. controls (6.95,-1.4) and (3.31,-0.3) .. (0,0) .. controls (3.31,0.3) and (6.95,1.4) .. (10.93,3.29)   ;
	%Straight Lines [id:da1071446202808346] 
	\draw    (417,410) -- (417,493.75) ;
	\draw [shift={(417,495.75)}, rotate = 270] [color={rgb, 255:red, 0; green, 0; blue, 0 }  ][line width=0.75]    (10.93,-3.29) .. controls (6.95,-1.4) and (3.31,-0.3) .. (0,0) .. controls (3.31,0.3) and (6.95,1.4) .. (10.93,3.29)   ;
	%Straight Lines [id:da8343013550865788] 
	\draw    (354,397.08) -- (287.8,397.08) ;
	\draw [shift={(285.8,397.08)}, rotate = 360] [color={rgb, 255:red, 0; green, 0; blue, 0 }  ][line width=0.75]    (10.93,-3.29) .. controls (6.95,-1.4) and (3.31,-0.3) .. (0,0) .. controls (3.31,0.3) and (6.95,1.4) .. (10.93,3.29)   ;
	%Straight Lines [id:da47963966492804444] 
	\draw    (493,398.08) -- (561.8,398.08) ;
	\draw [shift={(563.8,398.08)}, rotate = 180] [color={rgb, 255:red, 0; green, 0; blue, 0 }  ][line width=0.75]    (10.93,-3.29) .. controls (6.95,-1.4) and (3.31,-0.3) .. (0,0) .. controls (3.31,0.3) and (6.95,1.4) .. (10.93,3.29)   ;
	%Straight Lines [id:da9608392772530627] 
	\draw    (640.8,383.75) -- (476.36,252.5) ;
	\draw [shift={(474.8,251.25)}, rotate = 38.6] [color={rgb, 255:red, 0; green, 0; blue, 0 }  ][line width=0.75]    (10.93,-3.29) .. controls (6.95,-1.4) and (3.31,-0.3) .. (0,0) .. controls (3.31,0.3) and (6.95,1.4) .. (10.93,3.29)   ;
	%Straight Lines [id:da7018611079463521] 
	\draw    (641,413.5) -- (641,496.75) ;
	\draw [shift={(641,498.75)}, rotate = 270] [color={rgb, 255:red, 0; green, 0; blue, 0 }  ][line width=0.75]    (10.93,-3.29) .. controls (6.95,-1.4) and (3.31,-0.3) .. (0,0) .. controls (3.31,0.3) and (6.95,1.4) .. (10.93,3.29)   ;
	%Straight Lines [id:da22569959692434216] 
	\draw    (196,413.5) -- (196,497.75) ;
	\draw [shift={(196,499.75)}, rotate = 270] [color={rgb, 255:red, 0; green, 0; blue, 0 }  ][line width=0.75]    (10.93,-3.29) .. controls (6.95,-1.4) and (3.31,-0.3) .. (0,0) .. controls (3.31,0.3) and (6.95,1.4) .. (10.93,3.29)   ;
	%Straight Lines [id:da19101208980507967] 
	\draw    (189,380.5) -- (377.15,250.39) ;
	\draw [shift={(378.8,249.25)}, rotate = 145.34] [color={rgb, 255:red, 0; green, 0; blue, 0 }  ][line width=0.75]    (10.93,-3.29) .. controls (6.95,-1.4) and (3.31,-0.3) .. (0,0) .. controls (3.31,0.3) and (6.95,1.4) .. (10.93,3.29)   ;
	%Straight Lines [id:da3232508248817696] 
	\draw    (416,225.5) -- (416,157.75) ;
	\draw [shift={(416,155.75)}, rotate = 90] [color={rgb, 255:red, 0; green, 0; blue, 0 }  ][line width=0.75]    (10.93,-3.29) .. controls (6.95,-1.4) and (3.31,-0.3) .. (0,0) .. controls (3.31,0.3) and (6.95,1.4) .. (10.93,3.29)   ;
	%Straight Lines [id:da7712201822144013] 
	\draw    (643,526.5) -- (643,610.75) ;
	\draw [shift={(643,612.75)}, rotate = 270] [color={rgb, 255:red, 0; green, 0; blue, 0 }  ][line width=0.75]    (10.93,-3.29) .. controls (6.95,-1.4) and (3.31,-0.3) .. (0,0) .. controls (3.31,0.3) and (6.95,1.4) .. (10.93,3.29)   ;
	%Straight Lines [id:da533939043870303] 
	\draw    (419,528.5) -- (419,612.75) ;
	\draw [shift={(419,614.75)}, rotate = 270] [color={rgb, 255:red, 0; green, 0; blue, 0 }  ][line width=0.75]    (10.93,-3.29) .. controls (6.95,-1.4) and (3.31,-0.3) .. (0,0) .. controls (3.31,0.3) and (6.95,1.4) .. (10.93,3.29)   ;
	%Straight Lines [id:da8193744555220862] 
	\draw    (195,529) -- (195,612.75) ;
	\draw [shift={(195,614.75)}, rotate = 270] [color={rgb, 255:red, 0; green, 0; blue, 0 }  ][line width=0.75]    (10.93,-3.29) .. controls (6.95,-1.4) and (3.31,-0.3) .. (0,0) .. controls (3.31,0.3) and (6.95,1.4) .. (10.93,3.29)   ;
	%Straight Lines [id:da8922157214136073] 
	\draw    (416,128.5) -- (416,60.75) ;
	\draw [shift={(416,58.75)}, rotate = 90] [color={rgb, 255:red, 0; green, 0; blue, 0 }  ][line width=0.75]    (10.93,-3.29) .. controls (6.95,-1.4) and (3.31,-0.3) .. (0,0) .. controls (3.31,0.3) and (6.95,1.4) .. (10.93,3.29)   ;
	
	% Text Node
	\draw (581.64,304.53) node [anchor=north west][inner sep=0.75pt]  [rotate=-0.38]  {$\mathcal{D}$};
	% Text Node
	\draw (313.64,371.53) node [anchor=north west][inner sep=0.75pt]  [rotate=-0.38]  {$\mathcal{D}$};
	% Text Node
	\draw (509.89,371.48) node [anchor=north west][inner sep=0.75pt]  [rotate=-0.38]  {$\mathcal{\overline{D}}$};
	% Text Node
	\draw (389,311.4) node [anchor=north west][inner sep=0.75pt]    {$\mathcal{SH}( U)$};
	% Text Node
	\draw (365,386.4) node [anchor=north west][inner sep=0.75pt]    {$Cauchy\ formula$};
	% Text Node
	\draw (366,504.4) node [anchor=north west][inner sep=0.75pt]    {$S-func.\ cal.$};
	% Text Node
	\draw (113,389.48) node [anchor=north west][inner sep=0.75pt]    {$Harmonic\ int.\ formula$};
	% Text Node
	\draw (582,390.48) node [anchor=north west][inner sep=0.75pt]    {$Poly\ int.\ formula$};
	% Text Node
	\draw (151,508.9) node [anchor=north west][inner sep=0.75pt]    {$Q-func.\ cal$};
	% Text Node
	\draw (596,504.9) node [anchor=north west][inner sep=0.75pt]    {$P_{2} -func.\ cal.$};
	% Text Node
	\draw (330,228.48) node [anchor=north west][inner sep=0.75pt]    {$Fueter\ thm.\ int.\ formula$};
	% Text Node
	\draw (241.89,296.48) node [anchor=north west][inner sep=0.75pt]  [rotate=-0.38]  {$\mathcal{\overline{D}}$};
	% Text Node
	\draw (367,130.4) node [anchor=north west][inner sep=0.75pt]    {$F-func.\ cal.$};
	% Text Node
	\draw (341,619.9) node [anchor=north west][inner sep=0.75pt]    {$S-func.\ unbounded$};
	% Text Node
	\draw (568,618.9) node [anchor=north west][inner sep=0.75pt]    {$P_{2} -func.\ unbounded$};
	% Text Node
	\draw (652,554.9) node [anchor=north west][inner sep=0.75pt]    {$\breve{\phi ^{0}}( q) =\ \overline{\mathcal{D}}\left( q^{2} \phi ( q)\right)$};
	% Text Node
	\draw (435,561.9) node [anchor=north west][inner sep=0.75pt]    {$\phi ( q)$};
	% Text Node
	\draw (136,623.9) node [anchor=north west][inner sep=0.75pt]    {$Q-func.\ unbounded$};
	% Text Node
	\draw (202,559.9) node [anchor=north west][inner sep=0.75pt]    {$\tilde{\phi }( q) =\ \mathcal{D}( \phi ( q))$};
	% Text Node
	\draw (342.94,38.02) node [anchor=north west][inner sep=0.75pt]    {$F-Func.\ unbounded$};
	% Text Node
	\draw (429.15,89.53) node [anchor=north west][inner sep=0.75pt]    {$\breve{\phi }( q) =\ \Delta \left( q^{2} \phi ( q)\right)$};

\end{tikzpicture}
	}
\end{figure}
}
We observe that the previous diagram in the Clifford-setting is much more involved. Since the Fueter-Sce map is $T_{F2}= \Delta_{n+1}^{\frac{n-1}{2}}$, with $n$ being an odd number, there are different ways of factorizing $T_{F2}$ and these give rise to a more complicated fine structures, see \cite{Fivedim}.
\newline
\newline
In literature there is also another way to study a functional calculus for unbounded operators: the so-called direct approach. This is directly based on the Cauchy formula for unbounded domains. For the $S$-functional calculus and the Riesz-Dunford functional calculus this is consistent with the approach recalled in this paper in Section 1 and Section 2, see \cite{FJBOOK, RD}. In a forthcoming paper we aim to study the unbounded versions of the $F$-functional calculus, the $Q$-functional calculus and the $P_2$-functional calculus with the direct apporach. Moreover, we aim to show that these approaches are consistent with the ones presented in this paper.
\newline
\newline
The class of polyanalytic functions is widely studied in literature both in the complex case, see \cite{Balk}, and in the non-commutative setting, see \cite{Kamalpoly1, Kamalpoly2, ACDSPS, Brackx}. The motivations to consider this class of functions come from some elasticity problems studied by Kolossov and Muskhelishvili, see \cite{Russ}. Recently, this class of functions has been also related to the study of some time-frequency problems, see \cite{DMD2, AbreuFeic}. Moreover, some famous spaces of holomorphic functions have been expanded in the polyanalytic setting, see \cite{ACKS3, Bergmanvasi}.
\newline
\newline
In the following table
we summarize the conditions on the slice hyperholomorphic functions
given in order
to define the unbounded functional calculi of the
quaternionic fine structures:
	\newline
	\center{\small
	\begin{tabular}{| l | l |}
		\hline
		\rule[-4mm]{0mm}{1cm}
		{\bf Functional calculus} & {\bf Condition} \\
		\hline
		\rule[-4mm]{0mm}{1cm}
		$S$-functional & none\\
		\hline
		\rule[-4mm]{0mm}{1cm}
		$Q$-functional & none\\
		\hline
		\rule[-4mm]{0mm}{1cm}
		$P_2$-functional & $f(\alpha)=0$ and $\partial_{\alpha} f(\alpha)=0$\\
		\hline
		\rule[-4mm]{0mm}{1cm}
		$F$-functional & $f(\alpha)=0$\\
		\hline
	\end{tabular}
}

%\addcontentsline{toc}{chapter}{\bibname}
%\small
%\bibliographystyle{plain}
%\bibliography{Biblio}

\end{document}